\documentclass[onefignum,onetabnum]{siamart171218}

\usepackage{amsmath,amssymb,graphicx,psfrag}
\def\i{{\rm i}}

\def\C{{\mathbb C}}

\def\P{{\mathbb P}}
\def\R{{\mathbb R}}
\def\S{{\mathbb S}}

\def\E{{\mathbb E}}
\def\N{{\mathbb N}}

\def\D{\boldsymbol{D}}

\def\KK{{\mathcal K}}

\def\PP{{\mathcal P}}

\def\TT{{\mathcal T}}

\def\XX{{\mathcal X}}

\newcommand{\dist}[3][]{{\rm d\!l}[\ifthenelse{\equal{#1}{}}{}{#1;}#2,#3]}
\newcommand{\eff}[3][]{{\rm osc}\ifthenelse{\equal{#1}{}}{}{_{#1}}(#2;#3)}

\def\norm#1#2{\|#1\|_{#2}}

\def\set#1#2{\big\{#1\,:\,#2\big\}}

\def\eps{\varepsilon}

\def\y{\mathbf{y}}

\def\normL2#1#2{\|#1\|_{L^2(#2)}}
\newcommand{\dual}[3][]{#1\langle#2\,,\,#3#1\rangle}

\usepackage{lipsum}
\usepackage{amsfonts}
\usepackage{graphicx}
\usepackage{epstopdf}
\usepackage{algorithmic}
\ifpdf
  \DeclareGraphicsExtensions{.eps,.pdf,.png,.jpg}
\else
  \DeclareGraphicsExtensions{.eps}
\fi

\newsiamremark{remark}{Remark}
\newsiamremark{hypothesis}{Hypothesis}
\crefname{hypothesis}{Hypothesis}{Hypotheses}
\newsiamthm{claim}{Claim}

\headers{Improved Efficiency of a Multi-Index FEM}{J. Dick, M. Feischl, C. Schwab}

\title{Improved Efficiency of a Multi-Index FEM \\ for \\
       Computational Uncertainty Quantification\thanks{Submitted to the editors DATE.
\funding{Supported by the Australian Research Council (ARC) under
grant number DP150101770 (to JD) and DE170100222 (to MF), by the Swiss National Science Foundation
(SNSF) under grant number 200021\_159940 (to CS), and by the Deutsche
  Forschungsgemeinschaft (DFG) through CRC 1173 (to MF).}}}

\author{Josef Dick\thanks{School of Mathematics and Statistics,
         The University of New South Wales,
         Sydney 2052, Australia 
  (\email{josef.dick@unsw.edu.au}).}
\and Michael Feischl\thanks{TU Wien, Institute for Analysis and Scientific Computing, Wiedner Hauptstra\ss e 8-10, 1040 Wien.
  (\email{michael.feischl@tuwien.ac.at}).}
\and Christoph Schwab\thanks{SAM, ETH Z\"urich, ETH Zentrum HG G57.1, CH 8092 Z\"urich, Switzerland
  (\email{christoph.schwab@sam.math.ethz.ch }).}}

\usepackage{amsopn}

\ifpdf
\hypersetup{
  pdftitle={An Example Article},
  pdfauthor={D. Doe, P. T. Frank, and J. E. Smith}
}
\fi

\begin{document}

\maketitle 

\begin{abstract}
 We propose a multi-index algorithm for the Monte Carlo (MC) discretization of 
a linear, elliptic PDE with affine-parametric input. 
We prove an error vs. work analysis which allows a multi-level 
finite-element approximation in the physical domain, 
and apply the multi-index analysis with isotropic,
unstructured mesh refinement in the physical domain 
for the solution of the forward problem, for the approximation of the random field, 
and for the Monte-Carlo quadrature error.
Our approach allows Lipschitz domains and  
mesh hierarchies more general than tensor grids. 
The improvement in complexity over multi-level MC FEM is obtained from 
combining spacial discretization, dimension truncation and MC sampling
in a multi-index fashion.
Our analysis improves cost estimates compared to multi-level algorithms 
for similar problems and mathematically underpins the superior practical 
performance of multi-index algorithms for partial differential
equations with random coefficients.
\end{abstract}


\begin{AMS}
{\bf subject classification}
\end{AMS}

\maketitle
\section{Introduction}
The term \emph{multi-index Monte Carlo method} (MIMC for short) 
was first coined in the work~\cite{mi} 
as an extension of the multi-level Monte Carlo method 
(MLMC for short) developed in~\cite{giles}.
The MIMC idea abstracts sparse grids and sparse tensor products 
     to approximate  multivariate functions from sparse tensor products of 
     univariate hierarchic approximations in each variable, see 
     the surveys \cite{BunGriebSpG_Acta2004,SG11_518} and the references there.

Since the appearance of \cite{giles}, 
the multi-level idea has been applied in many areas
including high-dimensional integration, stochastic differential equations, and several
types of PDEs with random coefficients.
We refer to \cite{BSZ11_149,EigEtAl_AdMLMC_JUQ2016,GSU2016,SST2017}.
Most of these works addressed MLMC  algorithms, 
while \emph{multi-level quasi-Monte Carlo} (MLQMC for short) algorithms 
for PDEs with random field input data were addressed
only more recently in~\cite{KSSS17_2112,hoqmc,ml1,ml2}.
In the framework of PDEs with random coefficients, 
the idea of the multi-level approach is to introduce sequences of bisection
refined grids and to compute finite element (FE) approximations of a 
given partial differential equation (PDE) with random coefficients on each 
discretization level.
By varying the MC sample size on each level of the FE discretization
and by judicious combination of the individual approximations, 
it is possible to reduce the total cost (up to logarithmic factors) from
${\rm cost}(sampling)\times {\rm cost}(FEM)$ to ${\rm cost}(sampling)+{\rm cost}(FEM)$, 
where the individual cost terms are measured on the finest level. 
 
For example, in linear, elliptic PDEs in divergence form in a bounded domain $D$,
MLMC FEM were introduced in \cite{CGST2011,BSZ11_149}. 
It was shown there that MLMC FEM with continuous, piecewise affine (``$P_1$-FEM'')
finite elements in $D$ can provide a numerically computed
estimate of the mean field (or ``ensemble average'') of the random
solution $u$ (and, as explained in \cite{BSZ11_149}, also of its $2$- and $k$-point correlations)
which satisfies, in $H^1(D)$, essentially optimal (up to logarithmic terms) 
convergence rate bounds $O(h)$ in work which equals, 
in space dimension $d=2$, essentially $O(h^{-2})$.
These asymptotic orders equal the error vs. work relation 
for the solution of \emph{one instance of the corresponding deterministic problem}. 
In \cite{BSZ11_149}, the random input was assumed to consist only of 
a single term in a KL expansion of the random diffusion coefficient.
A similar result, again in space dimension $d=2$, 
for \emph{functionals $G(\cdot)\in H^{-1}(D)$ of the solution} 
was obtained in \cite{multilevel}. There, again $P_1$-FEM in $D$
were employed, but in order to achieve the higher FE convergence 
rate $O(h^2)$ for $G(\cdot)\in L^2(D)$, 
\emph{multi-level Quasi-Monte Carlo} integration over
the ensemble was necessary.

This idea was further extended in~\cite{mi} to include 
more than one parameter which is quantized into levels. 
One possible example for this approach, presented in~\cite{mi},
is to introduce anisotropic discretizations in the physical domain
(as, e.g., sparse grid FE discretizations) for which 
two (three) parameters control the element size in the coordinate direction. 
This `sparse grid' approach has been combined with a heuristic, adaptive algorithm
and a Quasi-Monte Carlo algorithm in~\cite{ami}. 
More examples of variations of this approach can be found in~\cite{mimi1,mimi2}.
In these approaches, the construction of sparse grid hierarchies in the
physical domain to access the multi-index efficiency could impose obstructions
on the shape of the physical domains which are amenable to this kind of
discretization.

In the present work, 
we follow a different (but, as we will show, very natural) 
approach: we  include the approximation 
of the random coefficients into the multi-index discretization and 
convergence analysis. As we show, this is effective due to the following
consideration: apart from toy problems, 
it is often not possible to obtain exact samples of the random coefficients. 
This is usually due to the fact that the random coefficient
is given in terms of some series expansion (Karhunen-Lo\`eve, Schauder, Wavelet,\ldots) for which only finitely many terms can be computed. While other works deal with the cost of computing the individual terms of the expansion (see, e.g.,~\cite{expcomp1,expcomp2}), the present work focuses on the necessary truncation of the series and hence assumes a cost of $\mathcal{O}(1)$ for each term.  
This particular approximation can constitute a major bottleneck in computations and it is therefore of practical importance to improve efficiency of algorithms.

Although the presently proposed approach is, in principle, more general,
we develop it here for 
affine-parametric random coefficients in a standard, linear
Poisson model problem
\begin{align}\label{eq:linpoisson}
 -{\rm div}(A \nabla u) = f\quad\text{in }D \;,
\qquad 
u=0\;\;\mbox{on}\;\; \partial D
\end{align}
for some Lipschitz domain $D\subseteq \R^d$.
We parametrize the uncertain diffusion coefficient, 
assumed to belong to $W^{1,\infty}(D)$, 
by a dimensionally truncated Karhunen-Lo\'eve expansion (``KL expansion'' for short), 
i.e., 
for given $x\in D$ and $\omega\in\Omega$ 
      (the probability space, see Section~\ref{sec:abstract})
\begin{align*}
A(x,\omega)
= \phi_0(x) + \sum_{j=1}^\infty \phi_j(x)\psi_j(\omega_j)
\approx 
A^\nu(x,\omega) := \phi_0(x) + \sum_{j=1}^{s_\nu}\phi_j(x)\psi_j(\omega_j),
\end{align*}
where $\{ s_\nu \}_{\nu \in\N}\subset \N$ is a strictly increasing sequence 
of ``dimension truncation'' parameters. 

%\todo{$\psi_j$ should be explained/ defined here, rather in \eqref{eq:KL0}}

Given a quantity of interest in terms of a linear functional $G(\cdot)$, 
the idea is to approximate the expectation of the 
exact solution $u$ of~\eqref{eq:linpoisson}, i.e., 
$\E(G(u))$ (where the expectation is taken over $\Omega$). 
This is done by computing several instances of the ``double difference'' 
$D_\ell^\nu = (u_\ell^\nu-u_{\ell-1}^\nu)-(u_\ell^{\nu-1}-u_{\ell-1}^{\nu-1})$, 
where $u_\ell^\nu$ denotes the FEM approximation of $u$ on
a mesh of size $h_\ell$ and with respect to the approximation $A^\nu$ of the exact 
(i.e. without truncation) random coefficient. 
As for any multi-level approach, this requires
a sequence of mesh-sizes $h_0,h_1,\ldots,h_\ell$ as introduced in Section~\ref{sec:fem}.
This leads to
\begin{align*}
 \E(G(u))\approx \sum_{0\leq \ell +\nu\leq N}Q_{m_{N-\ell-\nu}}(G(D_\ell^\nu)),
\end{align*}
where $Q_{m_{N-\ell-\nu}}$ denotes a 
MC sampling rule with given sample size $m_{N-\ell-\nu}\in\N$ such that $m_0 < m_1 < \cdots < m_N$.
The main result of this work is to prove that the above approximation is 
(up to logarithmic factors) optimal in the sense that 
it is as good as the approximation given by the naive approach 
$Q_{m_N}(G(u_N^N))$, where all components are computed on the finest level, while reducing the computational cost.

The error/cost estimates from Section~\ref{sec:mi} show that the distribution
of work among the individual levels is optimal up to logarithmic factors. 
This can be seen from the fact that the multi-index algorithm 
achieves the same (up to logarithmic factors) cost versus error ratio 
than the worst ratio of each of the involved algorithms 
(FEM, Monte Carlo, approximation of the random coefficient).
\subsection{Notation}
We will use the symbol $\lesssim$ to denote $\leq C$ for some multiplicative constant $C>0$ which does not depend on the parameters $\nu$, $\ell$, and $\omega$ unless stated otherwise.

%%%%%%%%%%%%%%%%%%%%%%%%%%%%%%%%%%%%%%%%%%%%%%%%%%%%%%%%%%%%%%%%%%%%%%%%%%%%%%%%
\section{Model problem}
\label{sec:ModProb}
%%%%%%%%%%%%%%%%%%%%%%%%%%%%%%%%%%%%%%%%%%%%%%%%%%%%%%%%%%%%%%%%%%%%%%%%%%%%%%%%
%
We chose a simple Poisson model problem to give a concise presentation 
of the ideas and proof techniques. 
The authors are confident that very similar techniques can be used
to include more general model problems. 
Moreover, we focus on the standard case of $H^2$-regularity of the Poisson problem. 
Intermediate cases with less regularity can be included
with the same arguments, but are left out for the sake of clarity.
\subsection{Abstract setting}\label{sec:abstract}

Consider a bounded ``physical domain'' $D\subseteq \R^d$ with Lipschitz boundary
in dimension $d\in\{2,3\}$. 
We model uncertain input data on  a probability space $(\Omega,\Sigma,\P)$. 
The mathematical expectation (``ensemble average'') w.r. to the probability measure $\P$ is denoted by $\E$.
               
Define the parametrized bilinear form 
%\todo{[Attn: $a(...)$ sollte deterministisch fuer ein $A$ definiert werden, und dann fuer
%       $A\to A(x,\omega)$ ausgewertet werden.]}

\begin{align*}
a(A;w,v):=\int_D A(x) \nabla w(x)\cdot \nabla v(x)\,dx
\end{align*}
for a scalar diffusion coefficient $A\colon D \to [0,\infty)$.
To model uncertain input data, we consider random diffusion coefficients
which satisfy $A(\cdot,\omega)\in L^\infty(D)$ for almost all $\omega\in\Omega$.
Precisely, $A$ is assumed a strongly measurable map from $(\Omega,\Sigma)$ 
to the Banach space $L^\infty(D)$, endowed with the Borel sigma algebra.
For $A\in L^\infty(D)$, the bilinear form $a(A;.,.)$ 
is continuous on $H^1_0(D)\times H^1_0(D)$, 
the usual Sobolev space given by

\begin{align*}
 H^1_0(D):=\set{v\in L^2(D)}{\nabla v\in L^2(D)^d,\,v|_{\partial D} =0}.
\end{align*}
We assume at hand a sequence of approximate diffusion coefficients
$(A^\nu)_{\nu\in\N}$ of $A=A^\infty$ 
which satisfy 
$A^\nu(\cdot,\omega)\in W^{1,\infty}(D)$ for almost all $\omega\in\Omega$
as well as
\begin{align}\label{eq:Aconv}
 \lim_{\nu\to\infty}\norm{A-A^\nu}{L^\infty(\Omega;W^{1,\infty}(D))}=0.
\end{align}
Furthermore, we assume 
the existence of deterministic bounds $A_{\rm min}$ and $A_{\rm max}$
such that for every $\nu\in\N\cup\{\infty\}$
\begin{align}\label{eq:minmax}
0<A_{\rm min}\leq \inf_{x\in D}A^\nu(x,\omega)
	\leq \sup_{x\in D}A^\nu(x,\omega) \leq A_{\rm max}<\infty.
\end{align}
To ease notation, we write  $a_\omega^\nu(\cdot,\cdot):=a(A^\nu(\omega),\cdot,\cdot)$.
Finally, suppose the right-hand side $f\in H^{-1}(D)$ (which is the dual space of $H^1_0(D)$). 
We embed $L^2(D)$ in $H^{-1}(D)$ via the compact
embedding $v\mapsto \dual{v}{\cdot}_D$ for all $v\in L^2(D)$.

The assumptions imply ellipticity and continuity of the bilinear form, i.e., 
for almost all $\omega\in\Omega$
\begin{align}\label{eq:ell}
\inf_{\nu\in\N\cup\infty}\inf_{w\in H^1_0(D)}\frac{a_\omega^\nu(w,w)}{\norm{w}{H^1(D)}^2}\geq A_{\rm min}
\end{align}
as well as
\begin{align}\label{eq:cont}
 \sup_{\nu\in\N\cup\infty}\sup_{w,v\in H^1_0(D)}\frac{a_\omega^\nu(w,v)}{\norm{w}{H^1(D)}\norm{v}{H^1(D)}}\leq A_{\rm max}.
\end{align}
In order to simplify some of the dependencies in the following, we assume $A_{\rm min}\leq 1$, which can be always achieved by scaling of $A$ and $f$.

The Lax-Milgram lemma 
implies with \eqref{eq:ell} and \eqref{eq:cont}
unique solvability and continuity of the solution operator.
This implies in particular the existence of a unique random solution $u$ (i.e. a strongly measurable map $u:\Omega \to H^1_0(D)$) 
which is defined pathwise by: 
given $\omega\in \Omega$, find $u(\omega) \in  H^1_0(D)$ such that

\begin{align*}
a(A(\omega);u(\omega),v)=\dual{f}{v}_D\quad\text{for all }v\in H^1_0(D),\; 
\text{ $\mathbb{P}$ a.e.}\; \omega\in\Omega.
\end{align*}

The Lipschitz continuity of the data-to-solution operator 
$S_A: A\to u$ (for fixed source term $f$) 
on the data $A\in L^\infty(D)$ such that \eqref{eq:minmax} holds
implies the strong measurability of $u: \Omega\to H^1_0(D)$.
We are interested in the expectation of a certain quantity of interest $G(\cdot)$
which is a deterministic, bounded linear functional 
$G(\cdot)\colon H^1_0(D)\to \R$, i.e.
\begin{align*}
 \E(G(u))\in \R.
\end{align*}
We assume that $G$ has an-$L^2$ representer, i.e., 
that there exists $g\in L^2(D)$ such that
\begin{align*}
G(v)=\int_D g v\,dx \quad\text{for all }v\in H^1_0(D).
\end{align*}
\subsection{Finite element discretization}\label{sec:fem}
\label{S:FEdisc}
We assume at our disposal
a sequence of nested triangulations $\{\TT_\ell\}_{\ell\in\N}$ 
with corresponding spaces $(\XX_\ell)_{\ell\in\N}$ 
(such that $\XX_\ell\subseteq \XX_k\subset H^1_0(D)$ for all $\ell\leq k$).
We assume the following approximation property of the spaces $\XX_\ell$: 
There exists a constant $C_{\rm approx}>0$ 
and a monotone sequence $\{h_\ell\}_{\ell\in\N}$ with $h_\ell>0$ and with $\lim_\ell h_\ell=0$
such that all $u\in H^{2}(D)$ satisfy
\begin{align}\label{eq:approx}
 \inf_{v\in \XX_{\ell}}\norm{u-v}{H^1(D)}\leq C_{\rm approx} h_\ell\norm{u}{H^2(D)}.
\end{align}
For convenience, we assume $h_{\ell+1}\geq C_{\rm unif} h_\ell$ for all $\ell\in\N$ 
and for some constant $C_{\rm unif}>0$. 
A popular example would be 
based on the nested sequence $\{ \TT_\ell \}_{\ell\geq 0}$ 
of regular, uniform triangulations of $D$
with corresponding decreasing sequence $\{ h_\ell \}_{\ell\geq 0}$ of mesh-widths 
$h_\ell = \max \{ {\rm diam}(T): T\in \TT_\ell \}$. 
The sequence $\{ \XX_\ell \}_{\ell \geq 0}$ of subspaces
can then be chosen as spaces of continuous, 
piecewise-linear functions on $\TT_\ell$.

\medskip

Given the sequence $\{ \XX_\ell \}_{\ell \geq 0}$ of subspaces,
the Galerkin approximation $u_\ell^\nu(\omega)\in \XX_\ell$ is
the solution of
\begin{align*}
 a_\omega^\nu(u_\ell^\nu(\omega),v)
	=
	\dual{f}{v}_D\quad\text{for all }v\in \XX_\ell\text{ and almost all }\omega\in\Omega.
\end{align*}
Unique solvability follows from the Lax-Milgram lemma and~\eqref{eq:ell}--\eqref{eq:cont}.
Consider the solution operators $\S_\ell^\nu(\omega)\colon H^{-1}(D)\to \XX_\ell$ 
defined by $\S_\ell^\nu(\omega) f:= u_\ell^\nu(\omega)$.
Moreover, let $(\S_\ell^\nu(\omega))^{-1}\colon \XX_\ell\to H^{-1}(D)$ 
be defined by
\begin{align*}
((\S_\ell^\nu(\omega))^{-1}u)(v):=a_\omega^\nu(u,v)\quad\text{for all } u\in\XX_\ell,\,v\in H^1_0(D).
\end{align*}
For brevity, we will omit the random parameter and just
write $\S_\ell^\nu:=\S_\ell^\nu(\omega)$. 
Moreover, we write $\S_\infty^\nu f:= u^\nu$, where $u^\nu(\omega)\in H^1_0(D)$ 
is the unique solution of
\begin{align*}
 a_\omega^\nu(u^\nu(\omega),v)=\dual{f}{v}_D\quad\text{for all }v\in H^1_0(D).
\end{align*}
Thus, $u^\nu$ denotes the exact solution corresponding to $A^\nu$ and 
$((\S_\infty^\nu(\omega))^{-1} \cdot ) (v):= a_\omega^\nu(\cdot,v) \in H^{-1}(D)$.

For simplicity of presentation,
we restrict to domains $D\subseteq \R^d$ which admit uniform (w.r. to all MC samples)
$H^2$-regularity of the exact solution as long as $f\in L^2(D)$:
there exists a constant $C_{\rm reg}>0$ such that for all $\omega\in\Omega$ 
and all $\nu\in\N$
\begin{align}\label{eq:reg}
\norm{\S_\infty^\nu f}{H^2(D)}
 \leq 
\frac{C_{\rm reg}}{A_{\rm min}^2}(1+\norm{A^\nu(\omega)}{W^{1,\infty}(D)})\norm{f}{L^2(D)}.
 \end{align}
%for all $f\in L^2(D)$. 

We remark that when the solution of the Poisson equation is $H^2$-regular, 
\eqref{eq:reg} follows as an immediate consequence.

Possible examples of domains $D$ which satisfy this property include 
domains with $C^2$-boundary $\partial D$ or convex domains.
%Note to ourselves: convex domains do have a Lipschitz boundary.

% 
 \begin{lemma}\label{lem:Sl}
  The discrete solution operators 
  $\S_\ell^\nu\colon H^{-1}(D)\to \XX_\ell$ as defined above satisfy for 
  almost all $\omega\in\Omega$ that
  \begin{align*}
   \norm{\S_\ell^\nu }{H^{-1}(D)\to H^1(D)}\leq A_{\rm min}^{-1}
  \end{align*}
as well as
 \begin{align*}
   \norm{(\S_\ell^\nu)^{-1} }{\XX_\ell\to H^{-1}(D)}\leq A_{\rm max}.
  \end{align*}
 \end{lemma}
\begin{proof}
The result follows immediately from~\eqref{eq:ell}--\eqref{eq:cont}.
\end{proof}

\section{Product structure of the approximation error}\label{sec:proderr}
The main purpose of this section is to prove the product error estimate of Theorem~\ref{thm:proderr} 
below at the end of this section.
This error estimate factors the total error into error contributions of the approximation 
of the random coefficient $A\approx A^\nu$ and finite element approximation error $h_\ell\to 0$.
We will restate several well-known results from finite-element analysis, 
as we will make use of the exact dependence on the constants. 

In view of the multi-index decomposition in Section~\ref{sec:mi}, 
we consider the ``difference of differences''
\begin{align}\label{eq:dd}
 D_\ell^\nu:= (u_\ell^\nu-u_{\ell-1}^\nu) - (u_\ell^{\nu-1}-u_{\ell-1}^{\nu-1})\colon \Omega\to\XX_\ell.
\end{align}
The goal is to get an error estimate of product form, as this will allow us 
to obtain nearly optimal complexity estimates.
The key observation is that the definition of $D_\ell^\nu$ and $\S_\ell^\nu$ implies that
\begin{align*}
 D_\ell^\nu &= (\S_\ell^\nu-\S_{\ell-1}^\nu)f - (\S_\ell^{\nu-1}-\S_{\ell-1}^{\nu-1})f\\
 &= (\S_\ell^\nu-\S_{\ell-1}^\nu)(\S_\ell^\nu)^{-1}(\S_\ell^\nu-\S_{\ell}^{\nu-1})f + \text{remainder},
\end{align*}
where the remainder term will be controlled in Lemma~\ref{lem:proderr0}, below. 
The product form of the first term already suggest the product error estimate which is the goal of this section.

In the following, we use the operator norm for bilinear forms $b(\cdot,\cdot)\colon \XX\times \XX\to \R$ 
for a Hilbert space $\XX$, i.e.,
\begin{align*}
 \norm{b}{} := \sup_{x,y\in\XX\setminus\{0\}}\frac{|b(x,y)|}{\norm{x}{\XX}\norm{y}{\XX}}.
\end{align*}
\begin{lemma}\label{lem:strang}
 Given $A,B\colon \Omega\to L^\infty(D)$, there holds the estimate
 \begin{align*}
   \norm{a(A(\omega),\cdot,\cdot)-a(B(\omega),\cdot,\cdot)}{}
         \leq  \norm{A(\omega)-B(\omega)}{L^\infty(D)} \quad\text{for all }\omega\in\Omega.
 \end{align*}
as well as 
\begin{align*}
	\norm{\S_\ell^\nu f-\S_\ell^\mu f}{H^1(D)}\leq 
	A_{\rm min}^{-2}\norm{A^\nu(\omega)-A^\mu(\omega)}{L^\infty(D)}\norm{f}{L^2(D)}
\end{align*}
for all $\ell,\nu,\mu\in\N$.
\end{lemma}
\begin{proof}
 The first estimate follows since we have for almost all $\omega\in\Omega$ that
 \begin{align*}
 |a(A(\omega),u,v) - a(B(\omega),u,v)|&\leq \int_D|A(x,\omega)-B(x,\omega)||\nabla u||\nabla v|\,dx
  \\
 &\leq \norm{A(\omega)-B(\omega)}{L^\infty(D)}\norm{u}{H^1(D)}\norm{v}{H^1(D)}.
 \end{align*}
 For the second statement, we combine the above with~\eqref{eq:ell}, and Lemma~\ref{lem:Sl}, 
 to obtain
 \begin{align*}
 A_{\rm min}\norm{\S_\ell^\nu f-\S_\ell^\mu f}{H^1(D)}^2
 &\leq
 a_\omega^\nu(\S_\ell^\nu f-\S_\ell^\mu f,\S_\ell^\nu f-\S_\ell^\mu f)\\
 &=\dual{f}{\S_\ell^\nu f-\S_\ell^\mu f}_D-a_\omega^\nu(\S_\ell^\mu f,\S_\ell^\nu f-\S_\ell^\mu f)\\
 &=
 (a_\omega^\mu-a_\omega^\nu)(\S_\ell^\mu f,\S_\ell^\nu f-\S_\ell^\mu f)\\
 &\leq A_{\rm min}^{-1}
 \norm{A^\nu-A^\mu}{L^\infty(D)}\norm{f}{L^2(D)}\norm{\S_\ell^\nu f-\S_\ell^\mu f}{H^1(D)}
 \end{align*}
for all $\omega\in\Omega$. 
This concludes the proof.
\end{proof}

\begin{lemma}[Galerkin orthogonality]\label{lem:galorth}
 There holds Galerkin orthogonality for all $k,\ell\in\N\cup\{\infty\}$, $\nu\in\N$ 
 and all $f\in H^{-1}(D)$ in the form
 \begin{align*}
  a_\omega^\nu(\S_k^\nu f,v)=a_\omega^\nu(\S_\ell^\nu f,v)\quad\text{for all } v\in \XX_{\min\{\ell,k\}}
                                                        \text{ and all }\omega\in\Omega.
 \end{align*}
 Particularly, this implies $\S_\ell^\nu(\S_k^\nu)^{-1}={\rm id}_{\XX_k}$ for all $\ell\geq k$ and $k<\infty$.
\end{lemma}
\begin{proof}
By definition, we have
 \begin{align*}
  a_\omega^\nu(\S_k^\nu f,v)=\dual{f}{v}_D=a_\omega^\nu(\S_\ell^\nu f,v).
 \end{align*}
To see the second statement, note that for $v\in\XX_k$ and $w\in\XX_\ell$, 
there holds by definition of the inverse
\begin{align*}
a_\omega^\nu(\S_\ell^\nu(\S_k^\nu)^{-1}v,w)=((\S_k^\nu)^{-1}v)(w)=a_\omega^\nu(v,w).
\end{align*}
This and the positive definiteness of 
the bilinear form $a_\omega^\nu(\cdot,\cdot)$ conclude the proof.
\end{proof}

For the next lemma, we define the energy norm
\begin{align*}
 \norm{u}{\omega,\nu}:=(a_\omega^\nu(u,u))^{1/2}.
\end{align*}
Note that~\eqref{eq:ell}--\eqref{eq:cont} ensure 
$A_{\rm min}^{1/2}\norm{\cdot}{H^1(D)}
   \leq  \norm{\cdot}{\omega,\nu}
   \leq A_{\rm max}^{1/2}\norm{\cdot}{H^1(D)}$ 
for almost all $\omega\in\Omega$ and for all $\nu\in \N$.

There holds the following variant of C\'ea's lemma:
\begin{lemma}[C\'ea's lemma]\label{lem:cea}
 For $v\colon \Omega\to \XX_\ell$, $\omega\in\Omega$, and $k\leq \ell$, we have
 \begin{align*}
  \norm{(\S_{\ell}^\mu(\S_\ell^\mu)^{-1}-\S_{k}^{\mu}(\S_\ell^\mu)^{-1})v(\omega)}{H^1(D)}
  & \leq A_{\rm min}^{-1/2}\inf_{w\in \XX_k}\norm{v(\omega)-w}{\omega,\mu}\\
  &\leq A_{\rm min}^{-1/2}A_{\rm max}^{1/2}\inf_{w\in \XX_k}\norm{v(\omega)-w}{H^1(D)}.
 \end{align*}
\end{lemma}
\begin{proof}
For almost all $\omega\in\Omega$, Galerkin orthogonality guarantees for all $w\in \XX_k$
\begin{align*}
  a_\omega^\mu\big((\S_{\ell}^\mu(\S_\ell^\mu)^{-1}-\S_{k}^{\mu}(\S_\ell^\mu)^{-1})v&,(\S_{\ell}^\mu(\S_\ell^\mu)^{-1}-\S_{k}^{\mu}(\S_\ell^\mu)^{-1})v\big)\\
  &= a_\omega^\mu\big((\S_{\ell}^\mu(\S_\ell^\mu)^{-1}-\S_{k}^{\mu}(\S_\ell^\mu)^{-1})v,\S_{\ell}^\mu(\S_\ell^\mu)^{-1}v-w\big).
\end{align*}
Since $a_\omega^\nu$ is a scalar product with respective norm $\norm{\cdot}{\omega,\nu}$, 
we have
\begin{align*}
 a_\omega^\mu\big((\S_{\ell}^\mu(\S_\ell^\mu)^{-1}&-\S_{k}^{\mu}(\S_\ell^\mu)^{-1})v,\S_{\ell}^\mu(\S_\ell^\mu)^{-1}v-w\big)\\
 &\leq \norm{(\S_{\ell}^\mu(\S_\ell^\mu)^{-1}-\S_{k}^{\mu}(\S_\ell^\mu)^{-1})v}{\omega,\mu}\norm{\S_{\ell}^\mu(\S_\ell^\mu)^{-1}v-w}{\omega,\mu}.
\end{align*}
Norm equivalence 
$A_{\rm min}^{1/2}\norm{\cdot}{H^1(D)}\leq  \norm{\cdot}{\omega,\nu}\leq A_{\rm max}^{1/2}\norm{\cdot}{H^1(D)}$ uniformly in $\omega$
and the fact that $\omega$ was arbitrary conclude the proof.
\end{proof}

The following lemma bounds the difference of the Galerkin projections $\S_{k}^\nu(\S_\ell^\nu)^{-1}$ for different parameters $\nu$. 
\begin{lemma}\label{lem:projection}
 There holds for $\ell,k,\nu,\mu\in\N$, all $v\colon \Omega\to \XX_\ell$, and all $\omega\in\Omega$
 \begin{align*}
  \norm{(\S_{k}^\nu(\S_\ell^\nu)^{-1}&-\S_{k}^{\mu}(\S_\ell^\mu)^{-1})v(\omega)}{H^1(D)}\\
  &\leq C_{\rm proj}(\omega)
  \norm{(A^\nu-A^\mu)(\omega)}{L^\infty(D)}\inf_{w\in \XX_k}\norm{v(\omega)-w}{H^1(D)},
 \end{align*}
where $C_{\rm proj}(\omega):= A_{\rm min}^{-3/2}A_{\rm max}^{1/2}$.
\end{lemma}
\begin{proof}
For $k\geq \ell$, we have $\S_{k}^\nu(\S_\ell^\nu)^{-1}={\rm id}_{\XX_\ell}=\S_{k}^\mu(\S_\ell^\mu)^{-1}$ and thus the assertion holds trivially.
Assume $k< \ell$. 
Define $v_k:=(\S_{k}^\nu(\S_\ell^\nu)^{-1}-\S_{k}^{\mu}(\S_\ell^\mu)^{-1})v\colon \Omega\to \XX_\ell$.
 Ellipticity~\eqref{eq:ell} of $a_\omega^\nu(\cdot,\cdot)$ together with Galerkin orthogonality shows for $\omega\in\Omega$
 \begin{align*}
  A_{\rm min}\norm{v_k(\omega)}{H^1(D)}^2  &\leq
  a_\omega^\nu(v_k(\omega),v_k(\omega))=a_\omega^\nu((\S_{\ell}^\nu(\S_\ell^\nu)^{-1}-\S_{k}^{\mu}(\S_\ell^\mu)^{-1})v(\omega),v_k(\omega)).
 \end{align*}
Since $\S_{\ell}^\nu(\S_\ell^\nu)^{-1}={\rm id}_{\XX_\ell}=\S_{\ell}^\mu(\S_\ell^\mu)^{-1}$, we have
 \begin{align*}
  A_{\rm min}\norm{v_k(\omega)}{H^1(D)}^2  &\leq
  a_\omega^\nu((\S_{\ell}^\mu(\S_\ell^\mu)^{-1}-\S_{k}^{\mu}(\S_\ell^\mu)^{-1})v(\omega),v_k(\omega))\\
  &=
  a_\omega^\mu((\S_{\ell}^\mu(\S_\ell^\mu)^{-1}-\S_{k}^{\mu}(\S_\ell^\mu)^{-1})v(\omega),v_k(\omega))\\
  &\quad + (a_\omega^\nu-a_\omega^\mu)((\S_{\ell}^\mu(\S_\ell^\mu)^{-1}-\S_{k}^{\mu}(\S_\ell^\mu)^{-1})v(\omega),v_k(\omega)).
\end{align*}
The first term on the right-hand side above is zero due to Galerkin orthogonality. Therefore, we obtain
\begin{align}\label{eq:help1}
  \norm{v_k(\omega)}{H^1(D)}^2
  &\lesssim  A_{\rm min}^{-1}\norm{a_{\omega}^\nu-a_{\omega}^\mu}{}\norm{(\S_{\ell}^\mu(\S_\ell^\mu)^{-1}-\S_{k}^{\mu}(\S_\ell^\mu)^{-1})v(\omega)}{H^1(D)}
  \norm{v_k(\omega)}{H^1(D)}.  
\end{align}
As shown in Lemma~\ref{lem:strang}, there holds $\norm{a_{\omega}^\nu-a_{\omega}^\mu}{}\leq \norm{(A^\nu-A^\mu)(\omega)}{L^\infty(D)}$. Moreover, we have by C\'ea's lemma (Lemma~\ref{lem:cea})
\begin{align*}
 \norm{(\S_{\ell}^\mu(\S_\ell^\mu)^{-1}-\S_{k}^{\mu}(\S_\ell^\mu)^{-1})v(\omega)}{H^1(D)}\leq A_{\rm min}^{-1/2}A_{\rm max}^{1/2}\inf_{w\in \XX_k}\norm{v(\omega)-w}{H^1(D)}.
\end{align*}
This together with~\eqref{eq:help1} concludes the proof.
\end{proof}

For the statement of the next result, we
recall the definition of the double difference $D_\ell^\nu$ in~\eqref{eq:dd}.

\begin{lemma}\label{lem:proderr0}
 There holds for all $\omega\in\Omega$ and $\ell\geq 1$
 \begin{align}\label{eq:proderr0}
 \begin{split}
\norm{D_\ell^\nu(\omega)}{H^1(D)}
  &\leq \norm{(\S_\ell^\nu-\S_{\ell-1}^\nu)(\S_\ell^\nu)^{-1}(\S_\ell^\nu-\S_{\ell}^{\nu-1})f}{H^1(D)}\\
  &\qquad + 
  C_{\rm proj}(\omega)\norm{(A^\nu-A^{\nu-1})(\omega)}{L^\infty(D)}\inf_{v\in \XX_{\ell-1}}\norm{u_\ell^{\nu-1}(\omega)-v}{H^1(D)},
\end{split}
  \end{align}
  where $C_{\rm proj}>0$ is defined in Lemma~\ref{lem:projection}.
\end{lemma}
\begin{proof}
Elementary manipulation of~\eqref{eq:dd} together with $\S^\nu_\ell(\S_k^\nu)^{-1}={\rm id}_{\XX_k}$, $k\leq \ell$ from Lemma~\ref{lem:galorth} show
\begin{align*}
 D_\ell^\nu&= ((\S_\ell^\nu-\S_{\ell-1}^\nu)-(\S_\ell^{\nu-1}-\S_{\ell-1}^{\nu-1}))f\\
 &= (\S_\ell^\nu-\S_{\ell-1}^\nu)(\S_\ell^\nu)^{-1}(\S_\ell^\nu-\S_{\ell}^{\nu-1})f-(\S_{\ell-1}^\nu(\S_\ell^\nu)^{-1}\S_\ell^{\nu-1}-\S_{\ell-1}^{\nu-1})f.
\end{align*}
The last term on the right-hand side satisfies
\begin{align}\label{eq:inter}
\begin{split}
 \norm{(\S_{\ell-1}^\nu(\S_\ell^\nu)^{-1}\S_\ell^{\nu-1}&-\S_{\ell-1}^{\nu-1})f}{H^1(D)}\\
 &\leq 
  \norm{(\S_{\ell-1}^{\nu-1}(\S_\ell^{\nu-1})^{-1}\S_\ell^{\nu-1}-\S_{\ell-1}^{\nu-1})f}{H^1(D)}\\
  &\qquad+
   \norm{(\S_{\ell-1}^\nu(\S_\ell^\nu)^{-1}-\S_{\ell-1}^{\nu-1}(\S_\ell^{\nu-1})^{-1})\S_\ell^{\nu-1}f}{H^1(D)}.
   \end{split}
\end{align}
The first term on the right-hand side satisfies for all $v\in \XX_{\ell-1}$
\begin{align*}
a_\omega^\nu( (\S_{\ell-1}^{\nu-1}(\S_\ell^{\nu-1})^{-1}\S_\ell^{\nu-1}-\S_{\ell-1}^{\nu-1})f,v)=
a_\omega^\nu( (\S_{\ell}^{\nu-1}(\S_\ell^{\nu-1})^{-1}\S_\ell^{\nu-1}-\S_{\ell}^{\nu-1})f,v)=0
\end{align*}
and thus $\norm{(\S_{\ell-1}^{\nu-1}(\S_\ell^{\nu-1})^{-1}\S_\ell^{\nu-1}-\S_{\ell-1}^{\nu-1})f}{H^1(D)}=0$.
For the second term on the right-hand side of~\eqref{eq:inter}, Lemma~\ref{lem:projection} with $\mu=\nu-1$ and $k=\ell-1$ proves
\begin{align*}
 \norm{(\S_{\ell-1}^\nu(\S_\ell^\nu)^{-1}&-\S_{\ell-1}^{\nu-1}(\S_\ell^{\nu-1})^{-1})\S_\ell^{\nu-1}f}{H^1(D)}\\
 &\lesssim \norm{A^\nu(\omega)-A^{\nu-1}(\omega)}{L^\infty(D)}\inf_{v\in \XX_{\ell-1}}\norm{u_\ell^{\nu-1}(\omega)-v}{H^1(D)}.
\end{align*}
Altogether, this concludes the proof.
\end{proof}

The following result is well-known and we reprove it in our setting for the convenience of the reader.
\begin{lemma}[Aubin-Nitsche duality]\label{lem:aubnit}
For all $v\in H^1_0(D)$, there holds
\begin{align*}
\norm{v-\S_\ell^\nu (\S_\infty^\nu)^{-1} v}{L^2(D)}
\leq 
C_{\rm approx}\frac{C_{\rm reg}}{A_{\rm min}^2}(1+\norm{A^\nu(\omega)}{W^{1,\infty}(D)})
h_\ell \norm{v}{H^1(D)}.
\end{align*}

\end{lemma}
\begin{proof}
Let $\iota\colon L^2(D)\to H^{-1}(D)$ be the usual embedding via the $L^2(D)$-scalar product. 
Define $V:=v-\S_\ell^\nu(\S_\infty^\nu)^{-1}  v$. 
We have with Galerkin orthogonality and by symmetry of $a_\omega^\nu$ for all $w\in \XX_\ell$
 \begin{align*}
 \norm{v-\S_\ell^\nu(\S_\infty^\nu)^{-1} v}{L^2(D)}^2
&=a_\omega^\nu(\S_\infty^\nu\circ \iota(V),V)=a_\omega^\nu(\S_\infty^\nu\circ \iota(V)-w,V)
\\
   &\leq \norm{\S_\infty^\nu\circ \iota(V)-w}{H^1(D)}\norm{V}{H^1(D)}.
   \end{align*}
   Since $w\in\XX_\ell$ was arbitrary, we get with~\eqref{eq:approx} and~\eqref{eq:reg}
   \begin{align*}
    \norm{&v-\S_\ell^\nu(\S_\infty^\nu)^{-1} v}{L^2(D)}^2\\
    &\leq C_{\rm approx}h_\ell \norm{\S_\infty^\nu\circ \iota(V)}{H^2(D)}\norm{V}{H^1(D)}\\
   &\leq C_{\rm approx}\frac{C_{\rm reg}}{A_{\rm min}^2}
	   (1+\norm{A^\nu(\omega)}{W^{1,\infty}(D)})h_\ell\norm{v-\S_\ell^\nu(\S_\infty^\nu)^{-1}  v}{L^2(D)}\norm{V}{H^1(D)}.
 \end{align*}
	With Lemma~\ref{lem:Sl}, %\todo{[Details]}
	we show $\norm{V}{H^1(D)}\leq (1+A_{\rm min}^{-1}A_{\rm min})\norm{v}{H^1(D)}$ 
	and thus we conclude the proof.
\end{proof}

The following result bounds the first term on the right-hand side of the estimate in Lemma~\ref{lem:proderr0} 
by an error estimate in product form.

\begin{lemma}\label{lem:proderr}
There holds for all $\omega\in\Omega$
\begin{align*}
 \norm{(\S_\ell^\nu-\S_{\ell-1}^\nu)(&\S_\ell^\nu)^{-1}(\S_\ell^\nu-\S_{\ell}^{\nu-1})f}{H^1(D)}\\
 &\leq \widetilde C_{\rm prod}(\omega) h_{\ell}\norm{(A^\nu-A^{\nu-1})(\omega)}{W^{1,\infty}(D)}\norm{f}{L^2(D)},
\end{align*}
where $\widetilde C_{\rm prod}(\omega)\simeq C_{\rm unif}A_{\rm min}^{-5}A_{\rm max}^{1/2}
	       (1+\max_{i\in\{0,1\}}\norm{A^{\nu-i}(\omega)}{W^{1,\infty}(D)})^2>0$.
 \end{lemma}
\begin{proof}
First, C\'ea's lemma (Lemma~\ref{lem:cea}) shows for $v\colon \Omega\to\XX_\ell$ 
 \begin{align*}
  \norm{(\S_\ell^\nu-\S_{\ell-1}^\nu(\omega))(\S_\ell^\nu)^{-1}v}{H^1(D)}
  \leq 
  A_{\rm min}^{-1/2}\inf_{w\in \XX_{\ell-1}}\norm{v(\omega)-w}{\omega,\nu}.
 \end{align*}
Let $v:= (\S_\ell^\nu-\S_{\ell}^{\nu-1})f$ 
and choose $w:= \S_{\ell-1}^\nu(\S_\infty^\nu)^{-1} v$. 
Then, there holds with Galerkin orthogonality
$a_\omega^\nu(w,v-w)=a_\omega^\nu(v-\S_{\ell-1}^\nu(\S_\infty^\nu)^{-1} v,w)=0$ 
and hence
\begin{align*}
 \norm{v-w}{\omega,\nu}^2&=a_\omega^\nu(v,v-w)=a_\omega^\nu(u^\nu-\S_{\ell}^{\nu-1}f,v-w)\\
 &=a_\omega^{\nu-1}(u^\nu-\S_{\ell}^{\nu-1}f,v-w)+(a_\omega^\nu-a_\omega^{\nu-1})(u^\nu-\S_{\ell}^{\nu-1}f,v-w)\\
 &=a_\omega^{\nu-1}(u^\nu,v-w)-\dual{f}{v-w}_D+(a_\omega^\nu-a_\omega^{\nu-1})(u^\nu-\S_{\ell}^{\nu-1}f,v-w),
 \end{align*}
where we inserted and subtracted $a_\omega^{\nu-1}(\cdot,\cdot)$. 
This leads to
 \begin{align*}
 \norm{v-w}{\omega,\nu}^2 
 &=a_\omega^{\nu-1}(u^\nu,v-w)-a_\omega^{\nu}(u^\nu,v-w) + (a_\omega^\nu-a_\omega^{\nu-1})(u^\nu-\S_{\ell}^{\nu-1}f,v-w)
 \\
 &=-(a_\omega^\nu-a_\omega^{\nu-1})(\S_{\ell}^{\nu-1}f,v-w)
 \\
 &=-(a_\omega^\nu-a_\omega^{\nu-1})(u^{\nu-1},v-w)-(a_\omega^\nu-a_\omega^{\nu-1})(u^{\nu-1}_\ell-u^{\nu-1},v-w),
 \end{align*}
 where we used $\S_\ell^{\nu-1} f = u_\ell^{\nu-1}$ and 
 we added and subtracted the corresponding exact solution $u^{\nu-1}$.
 Using the definition of the bilinear forms as well as integration by parts, 
 the above reads 
 \begin{align*}
 \norm{v-w}{\omega,\nu}^2&=\int_D\big(\nabla(A^\nu-A^{\nu-1})\cdot\nabla u^{\nu-1} 
                          + (A^\nu-A^{\nu-1})\Delta u^{\nu-1}\big)(v-w)\,dx\\
 &\qquad -(a_\omega^\nu-a_\omega^{\nu-1})(u^{\nu-1}_\ell-u^{\nu-1},v-w)\\
 &\leq \norm{A^\nu-A^{\nu-1}}{W^{1,\infty}(D)}\norm{u^{\nu-1}}{H^2(D)}\norm{v-w}{L^2(D)} \\
 &\qquad + \norm{a_\omega^\nu-a_\omega^{\nu-1}}{}\norm{u^{\nu-1}_\ell-u^{\nu-1}}{H^1(D)}\norm{v-w}{H^1(D)}.
\end{align*}
Finally, Lemma~\ref{lem:aubnit} shows 
\begin{align*}
\norm{v-w}{L^2(D)}
  &\lesssim h_{\ell-1} A_{\rm min}^{-2}(1+\norm{A^{\nu-1}(\omega)}{W^{1,\infty}(D)})\norm{v}{H^1(D)}\\
  &\lesssim h_{\ell-1} A_{\rm min}^{-3}(1+\norm{A^{\nu-1}(\omega)}{W^{1,\infty}(D)})\norm{f}{L^2(D)},
\end{align*}
where the last estimate uses Lemma~\ref{lem:Sl}.
Assumption~\eqref{eq:approx}, together with the C\'ea lemma (Lemma~\ref{lem:cea}), implies
\begin{align*}
\norm{u^{\nu-1}_\ell-u^{\nu-1}}{H^1(D)}\lesssim 
A_{\rm min}^{-1/2}A_{\rm max}^{1/2}h_\ell\norm{u^{\nu-1}}{H^2(D)}.
\end{align*}
Assumption~\eqref{eq:reg} implies 
\begin{align*}
 \norm{u^{\nu-1}}{H^2(D)}\lesssim  A_{\rm min}^{-2}(1+\norm{A^{\nu-1}(\omega)}{W^{1,\infty}(D)}\norm{f}{L^2(D)}
\end{align*}
and thus concludes the proof.

\end{proof}

Finally, we have collected all the ingredients to obtain the combined discretization error estimate
in product form.

\begin{proposition}\label{prop:proderr}
 There holds for all $\omega\in\Omega$
 \begin{align*}
  \norm{D_\ell^\nu(\omega)}{H^1(D)}
  &\leq C_{\rm prod}(\omega) h_\ell\norm{(A^\nu-A^{\nu-1})(\omega)}{W^{1,\infty}(D)}\norm{f}{L^2(D)},
  \end{align*}
  where $C_{\rm prod}(\omega)\simeq \widetilde C_{\rm prod}(\omega)>0$ and $\widetilde C_{\rm prod}$ is defined in Lemma~\ref{lem:proderr}.
\end{proposition}
\begin{proof}
 The first term on the right-hand side of~\eqref{eq:proderr0} is bounded by Lemma~\ref{lem:proderr}. 
 For the second term, we use~\eqref{eq:approx} together with~\eqref{eq:reg} to obtain a similar bound.
 Finally, we exploit that $h_\ell\geq C_{\rm unif} h_{\ell-1}$ and conclude the proof.
\end{proof}

Since we are interested in the error of the goal functional $G(\cdot)$, 
we may exploit a standard Aubin-Nitsche duality argument to double the rate of convergence.

\begin{theorem}\label{thm:proderr}
 There holds for all $\omega\in\Omega$
 \begin{align*}
  |G(D_\ell^\nu(\omega))|
  &\leq \overline{C}_{\rm prod}(\omega) h_\ell^2\min\big\{1,\norm{(A^\nu-A^{\nu-1})(\omega)}{W^{1,\infty}(D)}\big\}\norm{f}{L^2(D)}
  \norm{g}{L^2(D)}
  \end{align*}
 with $\overline{C}_{\rm prod}(\omega)>0$ depending on $C_{\rm prod}(\omega)$ from Proposition~\ref{prop:proderr} via
	\\
 	$\overline{C}_{\rm prod}(\omega)
 	\simeq 
	A_{\rm min}^{-5}A_{\rm max}\norm{A^{\nu}(\omega)}{W^{1,\infty}(D)}\norm{A^{\nu-1}(\omega)}{W^{1,\infty}(D)}C_{\rm prod}(\omega)$.
\end{theorem}
\begin{proof}
Let $g^\nu\in H^1_0(\Omega)$ such that $G(\cdot)=a_\omega^\nu(\cdot,g^\nu)$ 
 (note that such a function always exists due to the ellipticity~\eqref{eq:ell} of $a_\omega^{\nu-1}$).
There holds for $v,w\in \XX_{\ell-1}$
\begin{align*}
 G(D_\ell^\nu)&=a_\omega^\nu(u_\ell^\nu-u_{\ell-1}^\nu, g^\nu) -a_\omega^{\nu-1}(u_\ell^{\nu-1}-u_{\ell-1}^{\nu-1}, g^{\nu-1})\\
&= a_\omega^\nu(u_\ell^\nu-u_{\ell-1}^\nu, g^\nu-v) -a_\omega^{\nu-1}(u_\ell^{\nu-1}-u_{\ell-1}^{\nu-1}, g^{\nu-1}-v),
\end{align*}
where we used Galerkin orthogonality (Lemma~\ref{lem:galorth}) to insert $v\in\XX_{\ell-1}$. 
Adding and subtracting of $a_\omega^\nu(\cdot,\cdot)$ leads to
\begin{align*}
G(D_\ell^\nu)&= a_\omega^\nu(u_\ell^\nu-u_{\ell-1}^\nu, g^\nu-v) -a_\omega^{\nu}(u_\ell^{\nu-1}-u_{\ell-1}^{\nu-1}, g^{\nu-1}-v)\\
&\qquad + (a_\omega^{\nu}-a_\omega^{\nu-1})(u_\ell^{\nu-1}-u_{\ell-1}^{\nu-1}, g^{\nu-1}-v)\\
&= a_\omega^\nu(u_\ell^\nu-u_{\ell-1}^\nu, g^{\nu-1}-v) -a_\omega^{\nu}(u_\ell^{\nu-1}-u_{\ell-1}^{\nu-1}, g^{\nu-1}-v)\\
&\qquad + (a_\omega^{\nu}-a_\omega^{\nu-1})(u_\ell^{\nu-1}-u_{\ell-1}^{\nu-1}, g^{\nu-1}-v) +  a_\omega^\nu(u_\ell^\nu-u_{\ell-1}^\nu, g^\nu-g^{\nu-1}-w),
\end{align*}
where we added and subtracted $a_\omega^\nu(u_\ell^\nu-u_{\ell-1}^\nu, g^{\nu-1})$ and inserted $w\in\XX_{\ell-1}$ 
using Galerkin orthogonality (Lemma~\ref{lem:galorth}).
Recalling the definition of $D_\ell^\nu$ in~\eqref{eq:dd}, we arrive at 
\begin{align*}
G(D_\ell^\nu)
&= a_\omega^\nu(D_\ell^\nu, g^{\nu-1}-v) 
 + (a_\omega^{\nu}-a_\omega^{\nu-1})(u_\ell^{\nu-1}-u_{\ell-1}^{\nu-1}, g^{\nu-1}-v)\\
&\qquad+  a_\omega^\nu(u_\ell^\nu-u_{\ell-1}^\nu, g^\nu-g^{\nu-1}-w).
 \end{align*}
Lemma~\ref{lem:strang} and the C\'ea lemma (Lemma~\ref{lem:cea}) 
together with~\eqref{eq:approx} and~\eqref{eq:reg} allows us to estimate
 \begin{align}\label{eq:thm1}
 \begin{split}
  |G(D_\ell^\nu)|&\lesssim A_{\rm max} \norm{D_\ell^\nu}{H^1(D)}\norm{g^{\nu-1}-v}{H^1(D)}\\
  &\qquad+ \norm{A^\nu-A^{\nu-1}}{L^\infty(D)}\norm{u_\ell^{\nu-1}-u_{\ell-1}^{\nu-1}}{H^1(D)}\norm{g^{\nu-1}-v}{H^1(D)}\\
  &\qquad +\norm{u_\ell^{\nu}-u_{\ell-1}^{\nu}}{H^1(D)}\norm{g^\nu-g^{\nu-1}-w}{H^1(D)}\\
  &\lesssim  A_{\rm max} \norm{D_\ell^\nu}{H^1(D)}\norm{g^{\nu-1}-v}{H^1(D)} \\
  &\qquad + A_{\rm min}^{-5/2}A_{\rm max}^{1/2}(1+\norm{A^{\nu-1}(\omega)}{W^{1,\infty}(D)})\norm{f}{L^2(D)}h_\ell\\
  &\qquad \Big(\norm{A^\nu-A^{\nu-1}}{L^\infty(D)}\norm{g^{\nu-1}-v}{H^1(D)} +\norm{g^\nu-g^{\nu-1}-w}{H^1(D)}\Big).
 \end{split}
 \end{align}
Since $G(\cdot)=\int_D g(x)(\cdot)\,dx$ for some $g\in L^2(D)$, we obtain from~\eqref{eq:reg} that $g^\nu,g^{\nu-1}\in H^2(D)$. 
Therefore, and since $v\in \XX_{\ell-1}$ was arbitrary,~\eqref{eq:approx} and~\eqref{eq:reg} show
\begin{align*}
\inf_{v\in \XX_{\ell-1}} \norm{g^{\nu-1}-v}{H^1(D)}
 \lesssim A_{\rm min}^{-2}(1+\norm{A^{\nu-1}(\omega)}{W^{1,\infty}(D)})h_\ell \norm{g}{L^2(D)}.
\end{align*}
Moreover, there holds for all $v\in H^1_0(D)$
\begin{align*}
 a^\nu_\omega(g^\nu-g^{\nu-1},v)&= \dual{g}{v}_D - a_\omega^\nu(g^{\nu-1},v) = (a^{\nu-1}-a^\nu)(g^{\nu-1},v)\\
 &=\int_D\big(\nabla(A^\nu-A^{\nu-1})\cdot\nabla g^{\nu-1} + (A^\nu-A^{\nu-1})\Delta g^{\nu-1}\big)v\,dx.
\end{align*}
It is easy to see that the right-hand side is of the form $\dual{r}{v}_D$ for some $r\in L^2(D)$ with
\[
 \norm{r}{L^2(D)} \leq 2\norm{A^\nu-A^{\nu-1}}{W^{1,\infty}(D)}\norm{g^{\nu-1}}{H^2(D)}
 \lesssim \norm{A^\nu-A^{\nu-1}}{W^{1,\infty}(D)}\norm{g}{L^2(D)}.
\]
Therefore,~\eqref{eq:reg} shows 
\begin{align*}
 \norm{g^\nu-g^{\nu-1}}{H^2(D)}
 \lesssim A_{\rm min}^{-2}(1+\norm{A^{\nu}(\omega)}{W^{1,\infty}(D)})
          \norm{A^\nu-A^{\nu-1}}{W^{1,\infty}(D)}\norm{g}{L^2(D)}.
\end{align*}
Since $w\in \XX_{\ell-1}$ in~\eqref{eq:thm1} was arbitrary, 
the same argument and~\eqref{eq:approx} show
\begin{align*}
\inf_{w\in \XX_{\ell-1}}  &\norm{g^\nu-g^{\nu-1}-w}{H^1(D)}\\
&\lesssim h_\ell A_{\rm min}^{-2}(1+\norm{A^{\nu}(\omega)}{W^{1,\infty}(D)})
               \norm{A^\nu-A^{\nu-1}}{W^{1,\infty}(D)}\norm{g}{L^2(D)}.
\end{align*}
Altogether, we conclude the proof by use of Proposition~\ref{prop:proderr}, the above estimates, and insertion in~\eqref{eq:thm1}.
The minimum in the statement follows from standard arguments which we will sketch briefly. 
There holds for all $v\in\XX_{\ell-1}$ 
\begin{align*}
	G(u_\ell^\nu-u_{\ell-1}^\nu)&=a_\omega^\nu(u_\ell^\nu-u_{\ell-1}^\nu, g^\nu)=a_\omega^\nu(u_\ell^\nu-u_{\ell-1}^\nu, g^\nu-v).
\end{align*}
As above, choosing $v=\S_\ell^\nu(\S_\infty^\nu)^{-1}g^\nu$ and Lemma~\ref{lem:cea} together with~\eqref{eq:approx} leads to
\begin{align*}
	|	G(u_\ell^\nu-u_{\ell-1}^\nu)|&\lesssim \norm{u_\ell^\nu-u_{\ell-1}^\nu}{H^1(D)}h_{\ell-1}\norm{g}{L^2(D)}\\
	&\lesssim h_{\ell-1}^2\norm{f}{L^2(D)}\norm{g}{L^2(D)}.
\end{align*}
This concludes the proof.
\end{proof}

\section{Approximation of the random coefficient}
\label{S:ApproxRndCoef}
This section gives two examples of how to choose the random coefficient 
$A(x,\omega)$ as well as the approximations $A^\nu(x,\omega)$ in terms of the KL-expansion.

%%%%%%%%%%%%%%%%%%%%%%%%%%%%%%%%%%%%%%%%%%%%%%%%%%%%%%%%%%%%%%%%%%
\subsection{KL expansion}\label{sec:KL0}
In this section, we assume $\Omega=[0,1]^\N$, and define $\omega=(\omega_i)_{i\in\N}$.
We assume that $A^\nu$ is of the form (recall that $(s_\nu)_{\nu\in\N}$ is strictly increasing)
\begin{align}\label{eq:KL0}
 A^\nu(x,\omega):= \phi_0(x) + \sum_{j=1}^{s_\nu}\psi_j(\omega_j)\phi_j(x)
\end{align}
for functions $\phi_j\in W^{1,\infty}(D)$ and $\psi_j\in L^\infty([0,1],[-C_\psi,C_\psi])$ for some fixed $C_\psi>0$. While the literature often deals with the uniform case $\psi_j(\omega):=\omega-1/2$ (see next subsection), we allow this slightly more general case.
We assume that the series converges absolutely in $W^{1,\infty}(D)$ for all $\omega\in \Omega$ and hence define
\begin{align*}
 A(x,\omega):= A^\infty(x,\omega) := \phi_0(x) + \sum_{j=1}^{\infty}\psi_j(\omega_j)\phi_j(x).
\end{align*}
Moreover, we assume that~\eqref{eq:minmax} holds.
\begin{theorem}\label{thm:proderrKL0}
Under the assumptions of the current section, 
there holds
	\begin{align}\label{eq:KLest0}
	\norm{G(D_\ell^\nu)}{L^\infty(\Omega)}
	&\leq C_{\rm KL} h_\ell^2\sum_{i=s_{\nu-1}+1}^{s_\nu}\norm{\phi_i}{W^{1,\infty}(D)}\norm{f}{L^2(D)}\norm{g}{L^2(D)}.
	\end{align}
The constant $C_{\rm KL}>0$ depends on $C_\psi$ but is independent of $\ell$, $\nu$, and $\omega$.
\end{theorem}
\begin{proof}
The estimate follows immediately by definition of $A^\nu$ and Theorem~\ref{thm:proderr}. 
\end{proof}
\subsection{KL expansion with uniform random variables}
\label{sec:KL1}
In many cases, it is possible to reduce~\eqref{eq:KL0} to the simplified form
\begin{align}\label{eq:KL}
  A^\nu(x,\omega):=\phi_0(x)+\sum_{j=1}^{s_\nu}\omega_j\phi_j(x),
\end{align}
where now $\Omega=[-1/2,1/2]^\N$ and 
$
{\rm ess}\inf_{x\in D} \phi_0(x) > 0 \;.
$
This means setting $\psi_j(\omega) := \omega-1/2$ in~\eqref{eq:KL0}.
\begin{remark}
Note that theoretically, the case from Section~\ref{sec:KL0} 
can always be reduced to the present case. 
However, in many cases, this requires the user to pre-compute all functions 
$\phi_j$ which is computationally impractical.
\end{remark}

It turns out that in this case, an improved version of Theorem~\ref{thm:proderr} 
(see Theorem~\ref{thm:proderrKL} at the end of this section) 
can be derived by arguments already used for quasi-Monte Carlo estimates
(see, e.g., the works~\cite{ml1,ml2} and the references therein).
Given a subset $\Omega^\prime\subseteq \prod_{j\in\N}\C$, 
we define for all $j\in\N$
\begin{align*}
\Omega^\prime_j
:=\set{\omega_j\in\C}{\exists \omega_i\in\C,\,i\in\N\setminus\{j\}
\text{ such that }\omega=(\omega_1,\omega_2,\ldots)\in\Omega^\prime}.
\end{align*}

\begin{lemma}\label{lem:analytic}
Assume that $\Omega^\prime\supseteq \Omega $ is such that all results of Section~\ref{sec:proderr} 
hold true with $\Omega^\prime$ instead of $\Omega$. This is particularly the case if the random coefficient remains uniformly bounded away from zero and infinity also in $\Omega^\prime$.
Then the map $F\colon \Omega^\prime_j \to \C$, $\omega_j \mapsto G(\S_\ell^\nu(\omega) f)$ 
is holomorphic for all $j\in\N$.
\end{lemma}
\begin{proof}
Along the lines of the difference quotient argument in \cite{CDS10_404}, we verify complex differentiability
of the parametric solutions.

Fix $j\in\N$. 
Given $z\in\C$, define $\omega+z\in \C^\N$ by $(\omega+z)_i=\omega_i$ for all 
$i\neq j$ and $(\omega +z)_j=\omega_j+z$.
Let $z$ be sufficiently small such that there exists 
$\eps\geq 2|z|$ with $B_\eps(\omega)\subseteq \Omega^\prime$.
By definition, we have for $v\in\XX_\ell$
 \begin{align*}
  0&= a_{\omega+z}^\nu(\S_\ell^\nu(\omega+z) f,v)-a_{\omega}^\nu(\S_\ell^\nu(\omega) f,v)\\
  &=\int_{D}(A^\nu(x,\omega+z)-A^\nu(x,\omega))\nabla\S_\ell^\nu(\omega+z) f\cdot \nabla v\,dx +
  a_{\omega}^\nu(\S_\ell^\nu(\omega+z) f-\S_\ell^\nu(\omega) f, v).
 \end{align*}
Let $g^\nu\in \XX_\ell$ denote the representer of $G(\cdot)|_{\XX_\ell}$ with respect to $a_{\omega}^\nu$. 
This and the above allows us to compute
\begin{align}\label{eq:analytic1}
\begin{split}
 &\frac{G(\S_\ell^\nu(\omega+z) f)-G(\S_\ell^\nu(\omega) f)}{z} = \frac{a_{\omega}^\nu(\S_\ell^\nu(\omega+z)   
   f-\S_\ell^\nu(\omega) f,g^\nu)}{z}\\
 &\qquad\qquad
  =-\int_{D}\frac{A^\nu(x,\omega+z)-A^\nu(x,\omega)}{z}\nabla\S_\ell^\nu(\omega+z) f\cdot \nabla g^\nu\,dx.
 \end{split}
\end{align}
Since $A^\nu$ is holomorphic, 
Cauchy's integral formula shows for $B_{\eps}(\omega_j)\subset \Omega^\prime_j$ 
that
\begin{align*}
 \Big|&\frac{A^\nu(x,\omega+z)- A^\nu(x,\omega)}{z}-\partial_{\omega_j}A^\nu(x,\omega)\Big|\\
 &\qquad= 
 \frac{1}{2\pi}\Big|\int_{\partial B_\eps(\omega_j)}
 \frac{1}{z}\Big(\frac{A^\nu(x,y)}{(y-(\omega_j+z))} - \frac{A^\nu(x,y)}{(y-\omega_j)}\Big)-\frac{A^\nu(x,y)}{(y-\omega_j)^2}\,dy\Big|
 \\
 &\qquad= 
 \frac{1}{2\pi}\Big|\int_{\partial B_\eps(\omega_j)}\frac{A^\nu(x,y)}{(y-\omega_j-z)(y-\omega_j)} -\frac{A^\nu(x,y)}{(y-\omega_j)^2}\,dy\Big|
 \\
 &\qquad= 
 \frac{1}{2\pi}\Big|\int_{\partial B_\eps(\omega_j)}\frac{A^\nu(x,y)z}{(y-\omega_j-z)(y-\omega_j)^2}\,dy\Big|
 \\
 &\qquad
  \lesssim \eps^{-2}\norm{A^\nu}{L^\infty(\Omega\times D)}|z|\;.
\end{align*}
This uniform convergence in $|z|$ together with Lemma~\ref{lem:strang} shows that 
passing to the limit $z\to 0$ in $\C$ in~\eqref{eq:analytic1} leads to
\begin{align*}
 \partial_{\omega_j}G(\S_{\omega}^\nu f)= -\int_{D}\partial_{\omega_j}A^\nu(x,\omega)\nabla\S_\ell^\nu(\omega) f\cdot \nabla g^\nu\,dx\in\C.
\end{align*}
This shows that $F$ is complex differentiable and thus holomorphic. 
\end{proof}

\begin{lemma}\label{lem:derivative}
Let $(\varrho_j)_{j\in\N}$ be a positive sequence such that 
\begin{align*}
 \Omega\subset \Omega^\prime:=\prod_{j\in\N} B_{1+\varrho_j}(0)
\end{align*}
and that all the results of Section~\ref{sec:proderr} hold true with $\Omega^\prime$ instead of $\Omega$.
	Given $\ell,\nu\in\N$, the map $F_\ell^\nu\colon \Omega \to \R$, $\omega\mapsto G(D_\ell^\nu(\omega))$ satisfies
 \begin{align*}
 &\frac{\norm{\partial_\omega^\alpha F_\ell^\nu}{L^\infty(\Omega)}}{\norm{f}{L^2(D)}\norm{g}{L^2(D)}}\\
 &\;\leq \begin{cases}0 & \hspace{-2mm}\sum_{i=s_\nu+1}^\infty \alpha_i>0,
\\
C_{\rm der}\frac{\alpha!h_\ell^2}{\prod_{i=1}^\infty \varrho_i^{\alpha_i}} 
	 \min\{1,\sup_{\omega\in \Omega^\prime}\norm{A^\nu-A^{\nu-1}}{W^{1,\infty}(D)}\} &\text{ else,}
\end{cases}
 % \norm{F}{\WW_{\nu}^\alpha}\leq C_{\rm QMC} h_\ell\sup_{\omega\in\Omega}\norm{\phi_\nu}{W^{1,\infty}(D)}\norm{f}{L^2(D)}.
 \end{align*}
for all multi-indices $\alpha\in\N^\N$ with $|\alpha|<\infty$. 
The constant $C_{\rm der}>0$ depends only on $C_{\rm prod}$ from Theorem~\ref{thm:proderr}.
 \end{lemma}
\begin{proof}
For brevity of presentation, we fix $\ell$ and $\nu$ and write $F:=F_\ell^\nu$.
Lemma~\ref{lem:analytic} shows that $F$ can be extended to a function $F\colon \Omega^\prime\to \C$, 
which is holomorphic in each coordinate $\omega_j$.
Moreover, Lemma~\ref{lem:strang} proves that $F$ is uniformly continuous in $\Omega$. 
Therefore, we obtain immediately by induction that $F$ satisfies the multidimensional analog of  Cauchy's integral formula
for all $\omega\in \Omega^\prime$
\begin{align*}
F(\omega)=(2\pi\i)^{-n}\int_{\partial B_{\eps_1}(\omega_{d_1})} \cdots \int_{\partial B_{\eps_n}(\omega_{d_n})}
\frac{F(z)}{(z_1-\omega_{d_1})\ldots (z_{n}-\omega_{d_{n}})}\,dz_1\ldots dz_n,
\end{align*}
where $(d_1,\ldots,d_n)\in\N^n$ contains exactly $n$ distinct dimensions and 
the parameters
$\eps_i>0$, $i=1,\ldots,n$ are chosen so small that the integration domains 
of the contour integrals above are contained in $\Omega^\prime$. 
This shows immediately that for any multi-index $\alpha\in\N_0^\N$ with $|\alpha|<\infty$, 
$\partial_\omega^\alpha F$ is holomorphic in each variable.
Iterated application of Cauchy's integral formula shows for all $\omega \in \Omega$ that
\begin{align*}
\partial_\omega^\alpha F(\omega) 
= 
\Big(\prod_{i=1\atop \alpha_i\neq 0}^\infty \frac{\alpha_i!}{2\pi\i}\Big) 
\int_{\prod\limits_{i=1\atop \alpha_i\neq 0}^\infty \partial B_{\varrho_i}(\omega_i)} 
     \;\frac{F(z)}{\displaystyle\prod_{i=1\atop \alpha_i\neq 0}^\infty (z_i-\omega_i)^{\alpha_{i}+1}}
\,dz\;.
\end{align*}
This shows immediately
\begin{align*}
|\partial_\omega^\alpha F(\omega)|
\leq  
\Big(\prod_{i=1\atop \alpha_i\neq 0}^\infty \frac{\alpha_i!}{2\pi}2 \pi \varrho_i^{-\alpha_i}\Big)
\norm{F}{L^\infty(\Omega^\prime)} 
\leq 
\alpha!\Big(\prod_{i=1}^\infty\varrho_i^{-\alpha_i}\Big)\norm{F}{L^\infty(\Omega^\prime)}.
\end{align*}
This and Theorem~\ref{thm:proderr} with $A^\nu(\omega)= \phi_0+\sum_{i=1}^\nu \omega_i\phi_i$ conclude the proof.
\end{proof}

\begin{lemma}\label{lem:derivative2}
Define for sufficiently small $\delta>0$
\begin{align*}
 \beta_i:= \frac{\norm{\phi_i}{W^{1,\infty}(D)}}{({\rm ess} \inf_{x\in D}\phi_0(x)-2\delta)}.
\end{align*}
Given $\ell,\nu\in\N$, 
the map $F\colon \Omega \to \R$, $\omega\mapsto G(D_\ell^\nu(\omega))$ 
satisfies 
 \begin{align*}
 \norm{\partial_\omega^\alpha F}{L^\infty(\Omega)}
 \leq \widetilde C_{\rm der} 
 \begin{cases}0 & \sum_{i=s_\nu+1}^\infty \alpha_i>0, 
  \\
  \Big(\prod_{i=1}^{s_\nu} \beta_i^{\alpha_i}\Big) h_\ell^2  \norm{f}{L^2(D)}\norm{g}{L^2(D)}  
 &\text{else,}                                                        
 \end{cases}
 % \norm{F}{\WW_{\nu}^\alpha}\leq C_{\rm QMC} h_\ell\sup_{\omega\in\Omega}\norm{\phi_\nu}{W^{1,\infty}(D)}\norm{f}{L^2(D)}.
 \end{align*}
for all multi-indices $\alpha\in\N_0^{\N}$ with $|\alpha|\leq 2$. 
The constant $\widetilde C_{\rm der}>0$ depends only on $C_{\rm der}$ from Lemma~\ref{lem:derivative}, $\delta$, and $(\phi_j)_{j\in\N}$.
\end{lemma}
\begin{proof}
 Given $\alpha\in\N^{\N_0}$ with $|\alpha|\leq 2$ an admissible
 sequence $(\varrho_j)_{j\in\N}$ in Lemma~\ref{lem:derivative} is, given $\eps>0$,
 \begin{align*}
  \varrho_j:=\begin{cases}(\inf_{x\in D}\phi_0(x)-2\delta)\alpha_j/2\norm{\phi_j}{W^{1,\infty}(D)}^{-1}
                                 & \text{for all }j\in\N\text{ with }\alpha_j>0,\\
                           \eps  &\text{for all }j\in\N\text{ with }\alpha_j=0.
             \end{cases}
 \end{align*}
 This sequence satisfies
 \begin{align*}
  \inf_{\omega_i \in B_{1+\varrho_i}(0): i\in\N}
  \Re \big(\phi_0+ \sum_{i=1}^\nu \omega_i \phi_i\big)
  \geq \phi_0-({\rm ess} \inf_{x\in D}\phi_0(x)-2\delta)
          -\eps\sum_{i=1}^\infty \norm{\phi_j}{L^\infty(D)}
  \geq \delta
 \end{align*}
for sufficiently small $\eps>0$ (here $\Re$ denotes the real part). 
Moreover, the term $\norm{\phi_0+ \sum_{i=1}^\nu \omega_i \phi_i}{W^{1,\infty}(D)}$ remains 
uniformly bounded in $\Omega^\prime:=\prod_{i=1}^\infty B_{1+\varrho_i}(0)$.
This ensures that $\Omega^\prime$ satisfies all the assumptions required for $\Omega$. Thus, 
all results of Section~\ref{sec:proderr} remain valid for $\Omega^\prime$ instead of $\Omega$.
In particular, the constant
 $C_{\rm prod}(\omega)$ from Theorem~\ref{thm:proderr} is uniformly bounded in $\omega\in \Omega^\prime$.
The affine-parametric map $\omega\mapsto A^\nu(x,\omega)$ 
is holomorphic in each coordinate in $\Omega^\prime$, with 
constant derivative
\begin{align*}
 \partial_{\omega_j} A^\nu(x,\omega)=\begin{cases}
                                      \phi_j(x) &\text{for } j\leq s_\nu,\\
                                      0 &\text{else.}
                                     \end{cases}
\end{align*}
Moreover, since $|\alpha|\leq 2$ there holds 
\begin{align*}
\prod_{i=1}^\infty \varrho_i^{-\alpha_i}\leq \prod_{i=1}^\infty \beta_i^{\alpha_i}.
\end{align*}
This, together with Lemma~\ref{lem:derivative} concludes the proof.
\end{proof}

\begin{lemma}\label{lem:aux}
Let $g\in L^\infty(\Omega)$ be sufficiently smooth and let $g$ depend only on the first 
$s\in\N$ dimensions, i.e., $\partial_{\omega_i} g=0$ for all $i>s$.
For $0\leq r\leq s$ and $x=(x_1,x_2,\ldots,x_s)\in\Omega^s$, define the function space
\begin{align*}
	\PP^s_r(\Omega):={\rm span}\set{f\in L^\infty(\Omega)}{f(x)=\sum_{i=r+1}^s \alpha(x_1,\ldots,x_r)x_i,\, \alpha(x_1,\ldots,x_r)\in\R}.
\end{align*}
Assume that $\omega\in \Omega$ with $\omega_i=0$ for all $i>r$ implies $g(\omega)=0$.
Then, there holds
\begin{align*}
	\norm{g(\omega)}{L^\infty(\Omega)}\leq 
	\sum_{i=r+1}^s\norm{\partial_{\omega_i}g}{L^\infty(\Omega)}.
\end{align*}
Moreover, there exists $g_0\in \PP_r^s(\Omega)$ such that
\begin{align*}
	\norm{g(\omega)-g_0(\omega)}{L^\infty(\Omega)}\leq \frac12
	\sum_{i=r+1}^s\sum_{j=r+1}^i\norm{\partial_{\omega_i}\partial_{\omega_j}g}{L^\infty(\Omega)}.
\end{align*}
\end{lemma}
\begin{proof}
Let $\omega\in \R^s$.
There holds
\begin{align*}
g(\omega)&=\underbrace{g(\omega_1,\ldots,\omega_r,0,\ldots)}_{=0} 
	+ \sum_{i=r+1}^s\int_0^{\omega_{i}}\partial_{\omega_{i}}g(\omega_1,\ldots,\omega_{i-1},t_i,0,\ldots)\,dt_i\\
	&=
	\sum_{i=r+1}^s\int_0^{\omega_{i}} \Big(\partial_{\omega_i}g(\omega_1,\ldots,\omega_r,0,\ldots)\\
	&\qquad \qquad+\int_0^{t_i}\partial_{\omega_i}^2 g(\omega_1,\ldots,\omega_{i-1},s_i,0,\ldots)\,ds_i\\
	&\qquad \qquad+
	\sum_{j=r+1}^{i-1}
	\int_0^{\omega_j}\partial_{\omega_j}\partial_{\omega_i}g(\omega_1,\ldots,\omega_{j-1},s_j,0,\ldots)\,ds_j\Big)dt_i.
	\end{align*}
	Since the first integrand on the right-hand side does not depend on $\omega_i$, the above implies
	\begin{align*}
		g(\omega)&=
		\sum_{i=r+1}^s \Big(\omega_i\partial_{\omega_i}g(\omega_1,\ldots,\omega_r,0,\ldots)\\
		&\qquad +\int_0^{\omega_{i}}\Big(\int_0^{t_i}\partial_{\omega_i}^2g(\omega_1,\ldots,\omega_{i-1},s_i,0,\ldots)\,ds_i\\
		&\qquad\qquad +
		\sum_{j=r+1}^{i-1}
		\int_0^{\omega_j}\partial_{\omega_j}\partial_{\omega_i}g(\omega_1,\ldots,\omega_{j-1},s_j,0,\ldots)\,ds_j\Big)dt_i\Big).
	\end{align*}
	Since there holds  $(\omega\mapsto \omega_i\partial_{\omega_i}g(\omega_1,\ldots,\omega_r,0,\ldots))\in\PP^s_r(\Omega)$ for all $i\geq r+1$, 
we conclude the proof.
\end{proof}

\begin{theorem}\label{thm:proderrKL}
Under the assumptions of the current section, there holds
\begin{align}\label{eq:KLest1}
	\norm{G(D_\ell^\nu)}{L^\infty(\Omega)}
	&\leq C_{\rm KL} h_\ell^2\sum_{i=s_{\nu-1}+1}^{s_\nu}\norm{\phi_i}{W^{1,\infty}(D)}\norm{f}{L^2(D)}\norm{g}{L^2(D)}.
\end{align}
Moreover, there exists $g_0\in \PP^{s_\nu}_{s_{\nu-1}}(\Omega)$ such that
\begin{align}\label{eq:KLest2}
\begin{split}
\norm{&G(D_\ell^\nu)-g_0}{L^\infty(\Omega)}\\
&\leq C_{\rm KL} h_\ell^2\sum_{i=s_{\nu-1}+1}^{s_\nu}\sum_{j=s_{\nu-1}+1}^{s_\nu}\norm{\phi_i}{W^{1,\infty}(D)}\norm{\phi_j}{W^{1,\infty}(D)}\norm{f}{L^2(D)}
\norm{g}{L^2(D)}.
\end{split}
\end{align}
The constant $C_{\rm KL}>0$ is independent of $\ell$, $\nu$, and $\omega$.
\end{theorem}
\begin{proof}
The first estimate~\eqref{eq:KLest1} follows from the definition of $A^\nu$ and Theorem~\ref{thm:proderr}. 
For~\eqref{eq:KLest2}, 
the map $g(\omega):=D_\ell^\nu(\omega)$ satisfies the requirements of Lemma~\ref{lem:aux} with $r=s_{\nu-1}$.
Hence, the result follows immediately from Lemma~\ref{lem:aux} and Lemma~\ref{lem:derivative2}.
\end{proof}
\section{Monte Carlo integration}\label{sec:quad}
This section discusses the Monte Carlo quadrature rules. 
The uniform KL-expansion case (Section~\ref{sec:KL1}) allows us to increase the order of convergence by symmetrization of the Monte Carlo rule.
This section defines the Monte Carlo integration for the case that the random coefficient 
is given by a KL-expansion as discussed in Sections~\ref{sec:KL0}--\ref{sec:KL1}.

We make the standard assumption that the functions $\phi_i$ from~\eqref{eq:KL} satisfy
\begin{align}\label{eq:KLsum}
	 \norm{\phi_j}{W^{1,\infty}(D)}\leq C_{\rm KL} j^{-r}\quad\text{for all }j\in\N
\end{align}
for some $r>1$.
\begin{lemma}\label{lem:KLerr1}
	Define the Monte Carlo rule
	\begin{align*}
			Q_M(g):=\frac{1}{M}\sum_{i=1}^{M}g(X^i)
	\end{align*}
for uniformly distributed i.i.d $X^i\in [-1/2,1/2]^{s_\nu}$.
	Then, under the assumptions of Section~\ref{sec:KL0} given $\ell,\nu\in\N$, 
        the function $F\colon \Omega \to \R$, $\omega\mapsto G(D_\ell^\nu(\omega))$ satisfies 
	\begin{align*}
		\sqrt{\E_{\rm MC}|\E(F)-Q_M(F)|^2}\leq C_{\rm MC} 
		s_{\nu-1}^{1-r} \frac{h_\ell^2}{\sqrt{M}}\norm{f}{L^2(D)}\norm{g}{L^2(D)}.
	\end{align*}	
	Here, $\E_{\rm MC}(\cdot)$ denotes integration over the combined probability spaces of the $X^i$, $i=1,\ldots,M$, 
        whereas $\E(\cdot)$ denotes integration over $\Omega_\nu$.
\end{lemma}
\begin{proof}
The statement follows immediately from the standard Monte Carlo error estimate, Theorem~\ref{thm:proderrKL0}, 
and the fact that $\sum_{j=s_{\nu-1}+1}^{s_{\nu}}j^{-r}\lesssim s_{\nu-1}^{1-r}$.
\end{proof}

By symmetrization of the Monte Carlo sequence, we are able to increase the order of convergence in the truncation parameter $\nu$.
\begin{lemma}\label{lem:KLerr2}
	Define the symmetric Monte Carlo rule 
	\begin{align*}
	Q_M(g):=\frac{1}{2M}\sum_{i=1}^{M}(g(X^i_1,\ldots,X^i_{s_\nu})+g(X^i_1,\ldots,X^i_{s_{\nu-1}},-X^i_{s_{\nu-1}+1},\ldots,-X^i_{s_\nu})),
	\end{align*}
	where the $X^i\in [-1/2,1/2]^{s_\nu}$ are i.i.d. and uniformly distributed. 
        Under the assumptions of Section~\ref{sec:KL1}, there holds $Q_M(g_0)=0$ for all $g_0\in\PP_{s_{\nu-1}}^{s_\nu}(\Omega)$.
	Moreover,  given $\ell,\nu\in\N$, the map $F\colon \Omega \to \R$, $\omega\mapsto G(D_\ell^\nu(\omega))$ satisfies 
	\begin{align*}
		\sqrt{\E_{\rm MC}|\E(F)-Q_M(F)|^2}\leq C_{\rm MC} 
		s_{\nu-1}^{2(1-r)} \frac{h_\ell^2}{\sqrt{M}}\norm{f}{L^2(D)}\norm{g}{L^2(D)}.
	\end{align*}	
Here, $\E_{\rm MC}(\cdot)$ denotes integration over the combined probability spaces of the $X_i$, $i=1,\ldots,2^m$, whereas $\E(\cdot)$ denotes
integration over $\Omega_\nu$.
\end{lemma}
\begin{proof}
	First, we notice that for $g_0\in \PP_{s_{\nu-1}}^1(\Omega)$,
	there holds 
	\begin{align*}
	g_0(X^i_1,\ldots,X^i_{s_\nu})=-g_0(X^i_1,\ldots,X^i_{s_{\nu-1}},-X^i_{s_{\nu-1}+1},\ldots,-X^i_{s_\nu}).
	\end{align*}
	Therefore, we have $Q_M(g_0)=0$ for all $g_0\in\PP_{s_{\nu-1}}^1(\Omega)$.
	Thus, the statement follows from the standard Monte Carlo error estimate and Theorem~\ref{thm:proderrKL}, 
        where we note with \eqref{eq:KLsum}
	\begin{align*}
	\sum_{i=s_{\nu-1}+1}^{s_\nu}&\sum_{j=s_{\nu-1}+1}^{s_\nu}\norm{\phi_i}{W^{1,\infty}(D)}\norm{\phi_j}{W^{1,\infty}(D)}
        \\
        &\lesssim 
        \sum_{i=s_{\nu-1}+1}^{\infty}\sum_{j=s_{\nu-1}+1}^\infty i^{-r}j^{-r}\lesssim (s_{\nu-1})^{2(-r+1)}.
	\end{align*}
\end{proof}

\section{Multi-Index error control}\label{sec:mi}
The multi-index decomposition allows us to exploit 
the product error estimates and, hence, 
to improve the complexity of the finite-element/Monte Carlo algorithm.
%%%%%%%%%%%%%%%%%%%%%%%%%%%%%%%%%%%%%%%%%%%%%%%%%%%%%%%%%%%%%%%%%%
\subsection{Complexity of MIMCFEM}
\label{sec:mimcfem}
To quantify the complexity, i.e., the error vs. work, of the presently proposed MIMCFEM, 
we rewrite the exact solution as 
($Q_{m}$ denotes one of the MC sample averages $Q_M$ from Section~\ref{sec:quad} with $M=2^m$
 samples)
\begin{align*}
 \E(G(u))&= \sum_{j=0}^\infty (Q_{m_j}-Q_{m_j-1})(G(u))\\
 &= \sum_{j=0}^\infty \sum_{\ell=0}^\infty (Q_{m_j}-Q_{m_{j-1}})(G(u_{\ell}-u_{\ell-1}))\\
  &= \sum_{j=0}^\infty \sum_{\ell=0}^\infty \sum_{\nu=0}^\infty (Q_{m_j}-Q_{m_{j-1}})(G(D_\ell^\nu)),
\end{align*}
where $m_j\in\N$ and $Q_{m_{-1}}:= 0$.
By truncation of the series, we achieve a sparse approximation, i.e., given $N\in\N$
\begin{align*}
 \E(G(u))\approx G_{N}&:= \sum_{0\leq j+\ell+\nu\leq N} (Q_{m_j}-Q_{m_j-1})G(D_\ell^\nu)=
 \sum_{0\leq \ell+\nu\leq N} Q_{m_{N-\ell-\nu}}(G(D_\ell^\nu)).
\end{align*}
Recall the expectation of the Monte Carlo integration $\E_{\rm MC}(\cdot)$  
     and the expectation over $\Omega$ denoted by $\E(\cdot)$.
We define two quantities to quantify the efficiency of the 
presently proposed method: 
the MC sampling error is defined by 
\begin{align*}
 E_N:= \sqrt{\E_{\rm MC}|\E(G(u))-G_{N}|^2}
\end{align*}
whereas the cost model is defined by
\begin{align*}
 C_N:= (\text{The number of computational operations necessary to compute }G_N)
\end{align*}
and obviously depends on the chosen method discussed below.

%\subsection{The KL-expansion case}
First, we establish the cost model. A standard FEM will ensure $h_\ell\simeq 2^{-\ell}$ which implies $\#\TT_\ell \simeq 2^{d\ell}$. 
We assume a linear iterative solver such that solving the sparse FEM system costs $\mathcal{O}(2^{d\ell})$.

Under the assumptions of Section~\ref{sec:KL0} and \ref{sec:KL1}, we assume that we can compute the bilinear forms
\begin{align*}
 a_j(v,w):=\int_D\phi_j(x)\nabla v(x)\nabla w(x)\,dx\quad\text{for all }v,w\in \XX_\ell
\end{align*}
exactly in $\mathcal{O}(\#\TT_\ell)$.
Depending on the truncation parameters $s_\nu$, 
we have to compute $s_\nu$ bilinear forms $a_j(\cdot,\cdot)$ to obtain
in the affine case
\begin{align*}
a_\omega^\nu(v,w)=\sum_{j=1}^{s_\nu} \omega_j a_j(v,w),
\end{align*}
resulting in a cost of $\mathcal{O}(2^{d\ell}s_\nu)$. 
Altogether, this yields
\begin{align*}
C_N\simeq \sum_{0\leq j+\ell+\nu\leq N} 2^{m_j} 2^{d\ell}s_\nu
\end{align*}
Using Lemma~\ref{lem:KLerr1} as well as linear operator notation for $\E(\cdot)$ and $Q_{m_j}$, 
we see that the multi-index error satisfies
 \begin{align*}
  E_N&=\E_{\rm MC}\Big(\Big|\sum_{N< j+\ell+\nu} (Q_{m_j}-Q_{m_{j-1}})G(D_\ell^\nu)\Big|^2\Big)^{1/2}\\
  &\leq \sum_{0\leq\ell+\nu} \E_{\rm MC}\big(|(\E-Q_{m_{\max\{0,N-\ell-\nu+1\}}})G(D_\ell^\nu)|^2\big)^{1/2}\\
  &\lesssim \norm{f}{L^2(D)}\norm{g}{L^2(D)}\sum_{0\leq  \ell+\nu} 2^{-m_{\max\{0,N-\ell-\nu+1\}}/2}2^{-2\ell}s_{\nu-1}^{1-r}.
 \end{align*}
An obvious choice of the parameters $s_\nu$ and $m_j$ is to balance 
the work spent on each of the two tasks such that the three error contributions 
(FEM-discretization error, truncation error, quadrature error) are of equal asymptotic order. 
We define
\begin{align*}
 m_j:= \lceil 4j\rceil\quad\text{and}\quad s_\nu := \lceil 2^{\frac{2\nu}{r-1}}\rceil.
\end{align*}
With this, we have
\begin{align}\label{eq:errKL}
\begin{split}
E_N&\lesssim 
\norm{f}{L^2(D)}\norm{g}{L^2(D)}
\sum_{0\leq  \ell+\nu} 2^{-2{\max\{0,N-\ell-\nu+1\}}}2^{-2\ell}2^{-2\nu}
\\
&\lesssim \norm{f}{L^2(D)}\norm{g}{L^2(D)}(N+1)^22^{-2N} 
\end{split}
\end{align}
as well as
\begin{align}\label{eq:costKL1}
C_N\simeq 
\sum_{0\leq j+\ell+\nu\leq N} 2^{4j} 2^{d\ell}2^{\frac{2\nu}{r-1}}
\lesssim 2^{\max\{4,d,\frac{2}{r-1}\}N}.
\end{align}
Using the symmetrized Monte Carlo rule from Lemma~\ref{lem:KLerr2}, 
we see that the multi-index error improves to
\begin{align*}
E_N
&\lesssim 
\norm{f}{L^2(D)}\norm{g}{L^2(D)}\sum_{0\leq  \ell+\nu} 
           2^{-m_{\max\{0,N-\ell-\nu+1\}}/2}2^{-2\ell}s_{\nu-1}^{2(1-r)}.
\end{align*}
As above, we balance the contributions by
\begin{align*}
m_j:= \lceil 4j\rceil\quad\text{and}\quad s_\nu := \lceil 2^{\frac{\nu}{r-1}}\rceil.
\end{align*}
With this, we obtain the same error estimate as for the plain
Monte Carlo rule~\eqref{eq:errKL}, but with an improved cost estimate of
\begin{align}\label{eq:costKL2}
C_N^{\rm symm}\lesssim 2^{\max\{4,d,\frac{1}{r-1}\}N}.
\end{align}

\subsection{Comparison to multi-level (quasi-) Monte Carlo FEM}
The main difference to multi-level Monte Carlo is 
that the present method can capitalize on the approximation of the random coefficient, 
whereas the multi-level method has to treat this term in an a-priori fashion. 
However, the multi-level method can exploit symmetry properties of the exact operator 
to improve the rate of convergence
in the approximation of the random coefficient, i.e., 
it achieves the same accuracy with a cost 
$\mathcal{O}(2^{\frac{1}{r-1}N})$ instead of $\mathcal{O}(2^{\frac{2}{r-1}N})$.
This is worked out in the quasi-Monte Carlo case in~\cite{hoqmc} 
but transfers verbatim to the Monte Carlo case. 
Therefore, the multi-level (quasi-) Monte Carlo method with the same
level structure as described in the previous section will achieve
a cost versus error relation given by 
(see~\cite[Theorem~12]{multilevel} with $p=q=1/r-\eps$ for all $\eps>0$ and $\tau=2$ in their notation)
\begin{align*}
 E_N^{\rm ML}\lesssim (N+1)^\alpha 2^{-2N} \quad\text{with}\quad C_N^{\rm ML} 
\lesssim  
2^{\max\{4\lambda,d\}N + \frac{1}{r-1}N},
\end{align*}
where $\alpha>0$ is a constant and $1/(2\lambda)$ for $\lambda \in (1/2,1]$
is the convergence rate of the QMC quadrature 
(with the Monte Carlo rate \emph{formally} corresponding here 
 to the choice $1/(2\lambda)=1/2$).
Comparing the above estimates with the error vs. work 
estimates for the MIMCFEM from Section~\ref{sec:mimcfem}, 
we aim to identify parameter regimes in which the presently proposed MIMCFEM 
improves over alternative multi-level methods in terms of asymptotic error versus cost. 
We observe that standard multi-index Monte Carlo 
improves the multi-level Monte Carlo in case that 
\begin{align*}
 \max\{4,d,\frac{2}{r-1}\}<\max\{4,d\} + \frac{1}{r-1}\quad\text{equivalent to}\quad 
  \max\{4,d\}>\frac{1}{r-1},
\end{align*}
i.e., when the sampling and the FEM computations dominate the approximation of the random coefficient.
We conclude that the symmetric multi-index Monte Carlo method from Lemma~\ref{lem:KLerr2} 
\emph{improves the multi-index Monte Carlo method for all parameter combinations}. 
For $\lambda \in (1/2,1]$,
\begin{align*}
\max\{4,d,\frac{1}{r-1}\}
<
\max\{4\lambda,d\} + \frac{1}{r-1}
\quad\text{equivalent to}\quad 4-4\lambda<\frac{1}{r-1}
\end{align*}
the presently proposed, symmetric multi-index Monte Carlo FE method even improves
in terms of error vs. work as compared to the
    first order multi-level quasi-Monte Carlo method 
based on e.g. a randomly shifted lattice rule as in \cite{multilevel}. 
%[Attn: Josef, plse. add suitable references here]
This setting represents the case when
the approximation of the random coefficient dominates the sampling and the FEM computations.

%%%%%%%%%%%%%%%%%%%%%%%%%%%%%%%%%%%%%%%%%%%%%%%%%%%%%%%%%%%%%%%%%%%%%%%%%%%%%%%%%%
\section{Extension of the MIFEM convergence to Reduced Regularity in $D$}
\label{sec:RedReg}
%%%%%%%%%%%%%%%%%%%%%%%%%%%%%%%%%%%%%%%%%%%%%%%%%%%%%%%%%%%%%%%%%%%%%%%%%%%%%%%%%%
Up to this point,
the presentation and the error vs. work 
analysis assumed ``full elliptic regularity'' for data and solutions
of the model problem in Section \ref{sec:ModProb}. 
Specifically, we assumed that
the random diffusion coefficient $A$ and the deterministic right hand side $f$ 
in \eqref{eq:linpoisson} belong to $W^{1,\infty}(D)$ and to $L^2(D)$, respectively.
This, together with the convexity of the domain $D$ and the homogeneous Dirichlet
boundary conditions is well known to ensure $\mathbb{P}$-a.s. that 
$u\in L^2(\Omega;H^2(D))$. 
This, in turn, implies first order convergence of
conforming $P_1$-FEM on regular, quasi uniform meshes, and 
second order (super)convergence for continuous linear functionals in $L^2(D)$.
\emph{These somewhat restrictive assumptions were made in order to present 
      the MIFEM approach in the most explicit and transparent way}.
The present MIFEM error analysis is, however, valid under more general assumptions,
which we now indicate. 

	Still considering conforming $P_1$-FEM on regular meshes of triangles,
	\emph{mixed boundary conditions} and non convex polygons $D$, the same results can be shown verbatim by the same line of argument, provided that the following modifications
	of the FE error analysis are taken into account:
	(i) \emph{elliptic regularity}: as is well-known, the $L^2-H^2$ regularity
	result which we used will, in general, cease to be valid for non convex $D$,
	or for mixed boundary value problems. A corresponding theory is available
	and uses weighted Sobolev spaces. 
	We describe it to the extent necessary for extending our 
	error analysis for conforming $P_1$-FEM.
	In polygonal domains $D\subset \mathbb{R}^2$, 
	\emph{weighted, hilbertian Kondrat'ev spaces of order $m\in \N_0$ with shift $a\in \R$}
	are defined by
	\begin{equation}\label{eq:Kondr}
	\KK^m_a(D) := \{ v:D\to \R | r_D^{|\alpha|-a}\partial^\alpha v \in L^2(D), |\alpha|\leq m \}
	\end{equation}
	In \eqref{eq:Kondr}, $\alpha\in \N_0^2$ denotes a multi-index and $\partial^\alpha$
	the usual mixed weak derivative of order $\alpha = (\alpha_1,\alpha_2)$.
	In these spaces, there holds the following regularity result \cite[Thm. 1.1]{BLN2017}.
	\begin{proposition}\label{prop:RegKondr}
	Assume that $D\subset\R^2$ is a bounded polygon with straight sides.
	In $D$ consider the Dirichlet problem \eqref{eq:linpoisson} with 
	random coefficient $A\in L^\infty(\Omega;W^{1,\infty}(D))$ satisfying
	\eqref{eq:minmax}. Then the following holds:
	\begin{enumerate}
	\item
	There exists $\eta>0$ such that for every $|a|<\eta$,
	and for every $f\in \KK^0_{a-1}(D)$, the unique solution $u\in H^1_0(D)$
	of \eqref{eq:linpoisson} belongs to $\KK^2_{a+1}(D)$.
	\item
	For every fixed $f\in \KK^0_{a-1}(D)$, the 
	data-to-solution map $\S: W^{1,\infty}(D) \to \KK^2_{a+1}(D): A\mapsto u$
	is analytic for every $|a|<\eta$.
	\item
	There exists a sequence $\{ \TT^\ell \}_{\ell \geq 0}$ of regular,
	simplicial triangulations with refinements towards the corners of $D$
	such that there holds the approximation property
	\begin{equation}\label{eq:Approx}
	\forall w\in \KK^2_{a+1}(D):
	\quad \inf_{v\in S^1(D;\TT^\ell)} \| w - v \|_{H^1(D)} 
	\leq 
	Ch_\ell \| w \|_{ \KK^2_{a+1}(D) } \;,
	\end{equation}
	where $h_\ell := \max\{ {\rm diam}(T): T\in \TT^\ell \}$ and
	$\XX_\ell = \#(\TT^\ell) \lesssim h_\ell^{-2}$.
	\end{enumerate} 
	\end{proposition}
	We refer to \cite[Thm. 1.1]{BLN2017} for the proof of items 1. and 2., and to 
	\cite{AdlerNistor2015,BacutaNistorZikatanov2005,afem} for a proof of item 3.;
	we note in passing that \cite{AdlerNistor2015,BacutaNistorZikatanov2005} cover
	so-called \emph{graded} meshes, whereas item 3. for nested, bisection-tree meshes as
	generated e.g. by adaptive FEM is proved in \cite{afem}.

	With Proposition \ref{prop:RegKondr} at hand, 
	the preceding MIFEM error analysis extends \emph{verbatim} to 
	the present, more general setting: the $H^2(D)$ regularity results
	for the forward problem as well as for the adjoint problem
	extend to $\KK^2_{a+1}(D)$, under the assumption $f,g \in \KK^0_{a-1}(D)$,
	and under identical assumptions on the random coefficient $A$. 
        The use of the Cauchy integral theorem in the weighted function space
        setting is justified by item 2. combined with the (obvious) observation
        that affine-parametric functions such as \eqref{eq:KL}
        depend analytically on the parameters $\omega_j$.

	We also note that other discretizations, such as the symmetric interior-penalty discontinuous Galerkin (SIPDG) FEM,
	admit corresponding error bounds on graded meshes including the 
	superconvergence error bound in $L^2(D)$ \cite{MSS18_2617}. 
	A corresponding MIFEM algorithm and error analysis with exactly
        the same error vs. work bounds could also be 
        obtained for SIPDG discretization of the forward problem.

	We finally mention that Proposition \ref{prop:RegKondr}
	also extends verbatim to homogeneous, mixed boundary conditions,
	to symmetric matrix-valued random diffusion coefficients 
	$A = (a_{ij})_{i,j=1,2}\in W^{1,\infty}(D;\R^{2\times2})$
	(the space $W^{1,\infty}(D)$ could even be slightly larger,
	admitting singular behavior near corners of $D$)
	and to higher orders $m \geq 2$ of differentiation, allowing
	for Lagrangean FEM of polynomial degree $p=m\geq 2$ on locally refined meshes 
	in $D$. 
	A precise statement of these regularity results is available in \cite[Thm. 4.4]{BLN2017}.

%%%%%%%%%%%%%%%%%%%%%%%%%%%%%%%%%%%%%%%%%%%%%%%%%%%%%%%%%%%%%%%%%%%%%%%%%%%%%%%%%%
\section{Numerical experiments}
\label{sec:NumEx}
%%%%%%%%%%%%%%%%%%%%%%%%%%%%%%%%%%%%%%%%%%%%%%%%%%%%%%%%%%%%%%%%%%%%%%%%%%%%%%%%%%
We provide numerical tests in space dimension $2$ to verify the theoretical
results. In the first example, we choose uniform mesh refinement in a convex 
domain $D$ and irregular forcing function $f$ (which is to say in the
present setting of first order FEM that $f\not \in L^2(D)$).
The second example will feature a non-convex domain with
re-entrant corner and sequences $\{ \TT_\ell \}_{\ell}$ of locally refined,
nested regular triangulations of $D$.
%%%%%%%%%%%%%%%%%%%%%%%%%%%%%%%%%%%%%%%%%%%%%%%%%%%%%%%%%%%%%%%%%%%%%%%%%%%%%%%%%
\subsection{Irregular forcing and uniform mesh refinement}
\label{sec:IrregLoad}
%%%%%%%%%%%%%%%%%%%%%%%%%%%%%%%%%%%%%%%%%%%%%%%%%%%%%%%%%%%%%%%%%%%%%%%%%%%%%%%%%%
For purposes of comparison, 
we use a similar example as in \cite[Section~5.2]{ml1}. 
We choose the convex domain $D=[0,1]^2$ and define the 
scalar random coefficient function $A$ by
\begin{align*}
A(x,\omega)&:=1/2+\sum_{k_1,k_2=1}^\infty \frac{\omega_{k_1,k_2}}{(k_1^2+k_2^2)^2}\sin(k_1\pi x_1)\sin(k_2\pi x_2)
\\
&:=1/2+\sum_{j=1}^\infty \frac{\omega_j}{\mu_j}\sin(k_{1,j}\pi x_1)\sin(k_{2,j}\pi x_2),
\end{align*}
where $\mu_j:=(k_{1,j}^2+k_{2,j}^2)^2$ such that $\mu_i\leq \mu_j$ for all $i\leq j$ 
and ties are broken in an arbitrary fashion.
This ensures that the $\phi_j$ satisfy~\eqref{eq:KLsum} with $r=2$. 
The variational form of the problem then reads
\begin{align*}
\mbox{Find} \;\;u\in H^1_0(D):\quad  a(A(\cdot,\omega);u,v)=f(v) \quad \forall v\in H^1_0(D)\;.
\end{align*}
where $f\in H^{-1/2-\eps}(D)$ for all $\eps>0$ is defined by 
$$
f(v):=\int_\Gamma v(x_1,x_2)x_1\,d\Gamma(x_1,x_2) 
     = \sqrt{2} \int_0^1 t \, v(t, 1-t) \,d t 
$$ 
for $\Gamma=\set{(0,1)+r(1,-1)}{0\leq r\leq 1 }$ being a diagonal of $D$.
Note that we choose the weight $x_1$ in the integral in the definition of the right-hand side 
to introduce some non-symmetric quantities and thus avoid super-convergence effects.
We consider the quantity of interest $G(u):=\int_{D^\prime} u\,dx$, 
where $D^\prime=(1/2,1)^2\subset D$.
Whereas the analysis of the present paper is focused 
on the full regularity case with right-hand side $f\in L^2(D)$, 
all arguments remain valid in case of reduced regularity of the right-hand side $f\in H^{-1/2-\eps}(\D)$ 
(for the case of reduced regularity due to re-entrant corners, see the second experiment).  

The finite element discretization is 
based on first order, nodal continuous, piecewise affine 
finite elements $\XX_\ell$ on a uniform partition of $[0,1]^2$ 
into $2^{2\ell+1}$ many congruent triangles 
(one example is shown in Figure~\ref{fig:mesh}). 
The mesh width of this triangulation is $h_\ell= \mathcal{O}(2^{-\ell})$. 
%note to ourselves: longest edge has length \sqrt{2} * w^{-\ell}
Note that the cost model applies as we can compute the stiffness matrix exactly
since the gradients of the shape functions are constants
and the anti-derivatives of products of sine functions are known over triangles. 
The error expected by theory for the FEM
on mesh-level $\ell$ is 
$\mathcal{O}(h_\ell)=\mathcal{O}(2^{-3/2\ell})$ 
(due to the reduced regularity of the right-hand side $f$).
Thus we choose the $m_j:=3j$ as well as 
$s_\nu=\lceil 2^{\nu/(r-1)}\rceil$ for the original algorithm and
$s_\nu=\lceil 2^{\nu/(2(r-1))}\rceil$ for the symmetrized version. 
Therefore we expect that the  errors for both algorithms satisfy 
$E_N=\mathcal{O}(2^{-3/2N})=\mathcal{O}(C_N^{-1/2})$, where 
%
%{\bf why is the cost $C_N$ additive?}
%\begin{align*}
% C_N=\#\text{FEM-dof} + \#\text{Monte Carlo samples} + \#\text{terms in approximation of }A 
%\end{align*}
$C_N$ as defined in \eqref{eq:costKL1}, \eqref{eq:costKL2}
    denotes the cost of the multi-index FEM on level $N$. 
This is confirmed in Figure~\ref{fig:convMI2d}.
For the numerical experiments, 
we compare with a reference solution computed with a higher-order 
Quasi-Monte Carlo method proposed in~\cite{ml1}. 
The reference value is computed with a higher order QMC rule.\footnote{
The authors thank F. Henriquez, a PhD student at the Seminar for Applied Mathematics of ETH, 
for computing the reference value.}

To smooth out the effects of MC sampling, the plotted relative errors are averaged over 
$20$ runs of the respective multi-index algorithm
(we also plot empirical 90\%-confidence intervals for each error point).

\begin{figure}
\centering 
\includegraphics[width=0.45\textwidth]{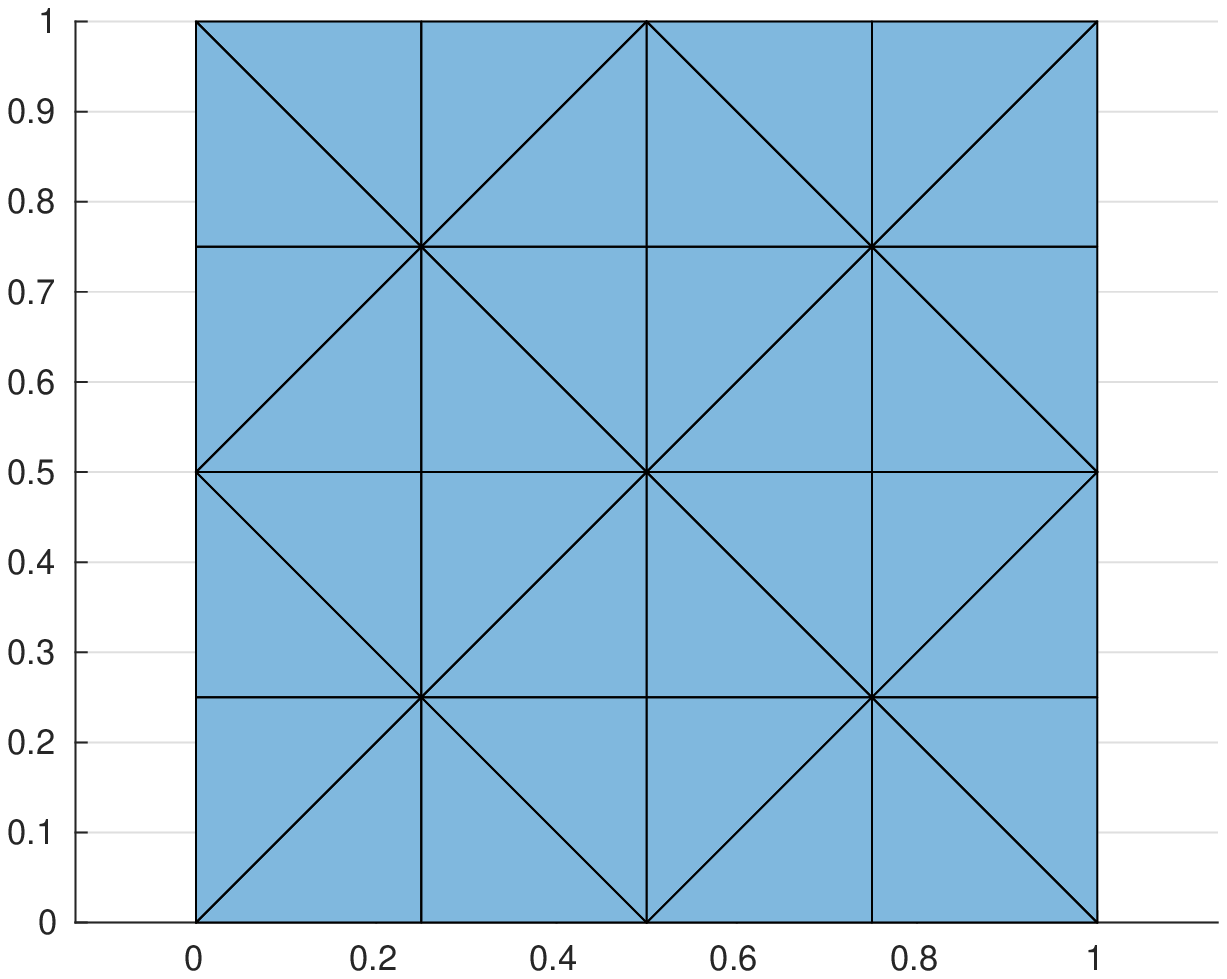}\qquad%
\includegraphics[width=0.45\textwidth]{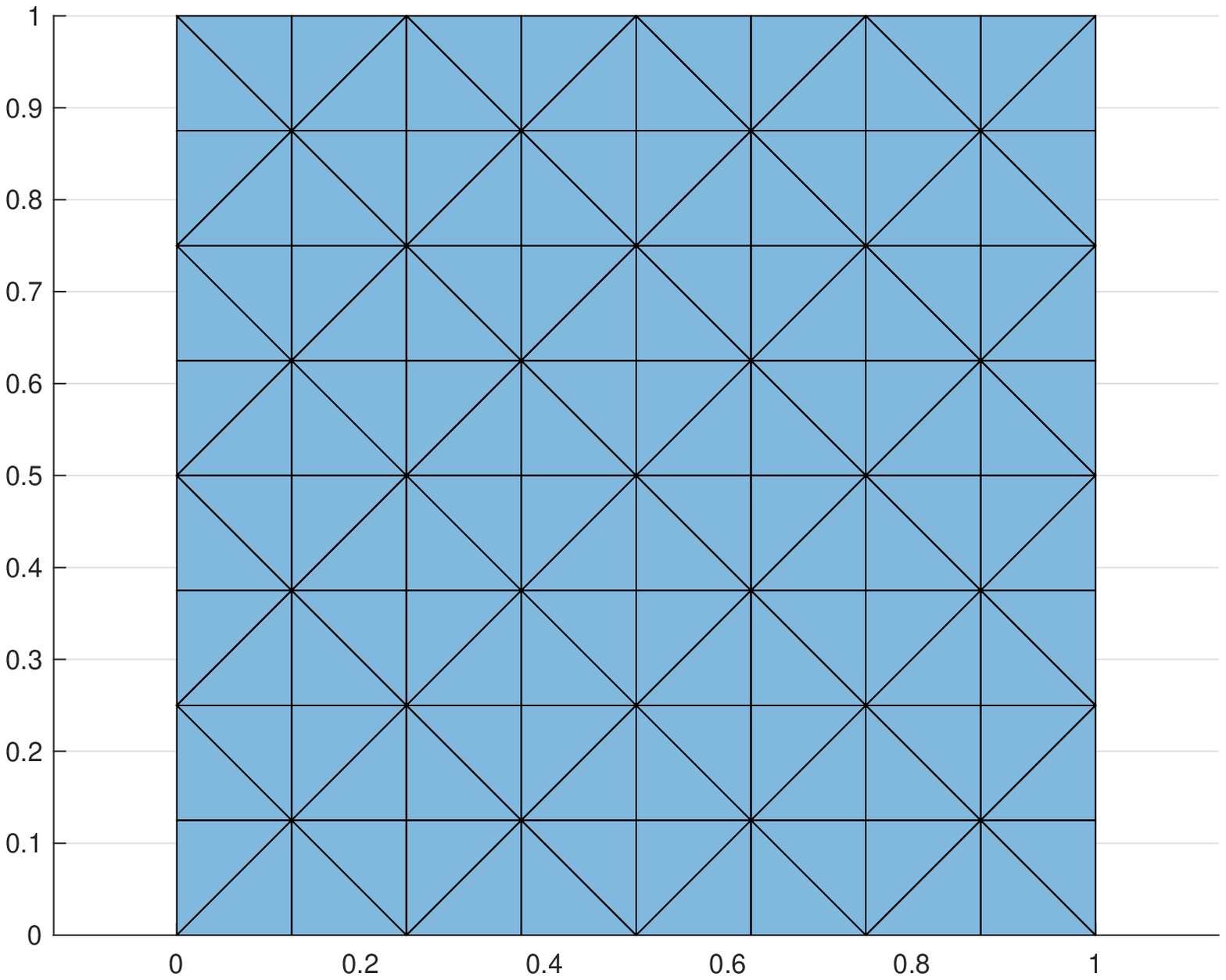}
 \caption{Two levels of mesh-refinement for the unit-square domain.}
 \label{fig:mesh}
\end{figure}

\begin{figure}
\centering
\psfrag{error}[cc][cc]{\tiny relative error}
\psfrag{N}[cc][cc]{\tiny cost of multi-index FEM $C_N$}
\psfrag{diff}{\tiny $|G-G_N|/G$}
\psfrag{ref}{\tiny $\mathcal{O}(C_N^{-1/2})$}
\includegraphics[width=0.45\textwidth]{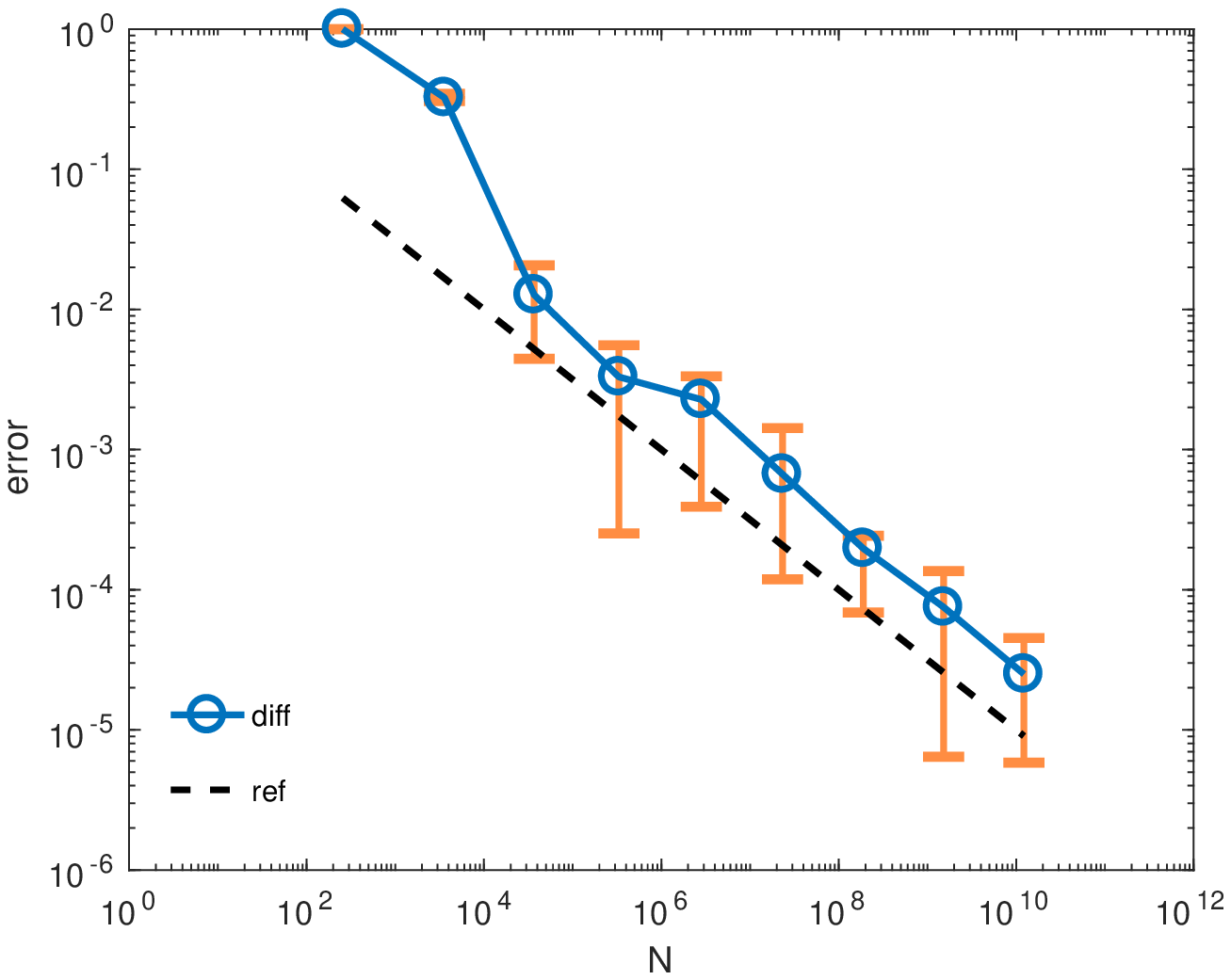}%
\psfrag{diff}{\tiny $|G-G_N^{\rm symm}|/G$}%
\hspace{5mm}
\includegraphics[width=0.45\textwidth]{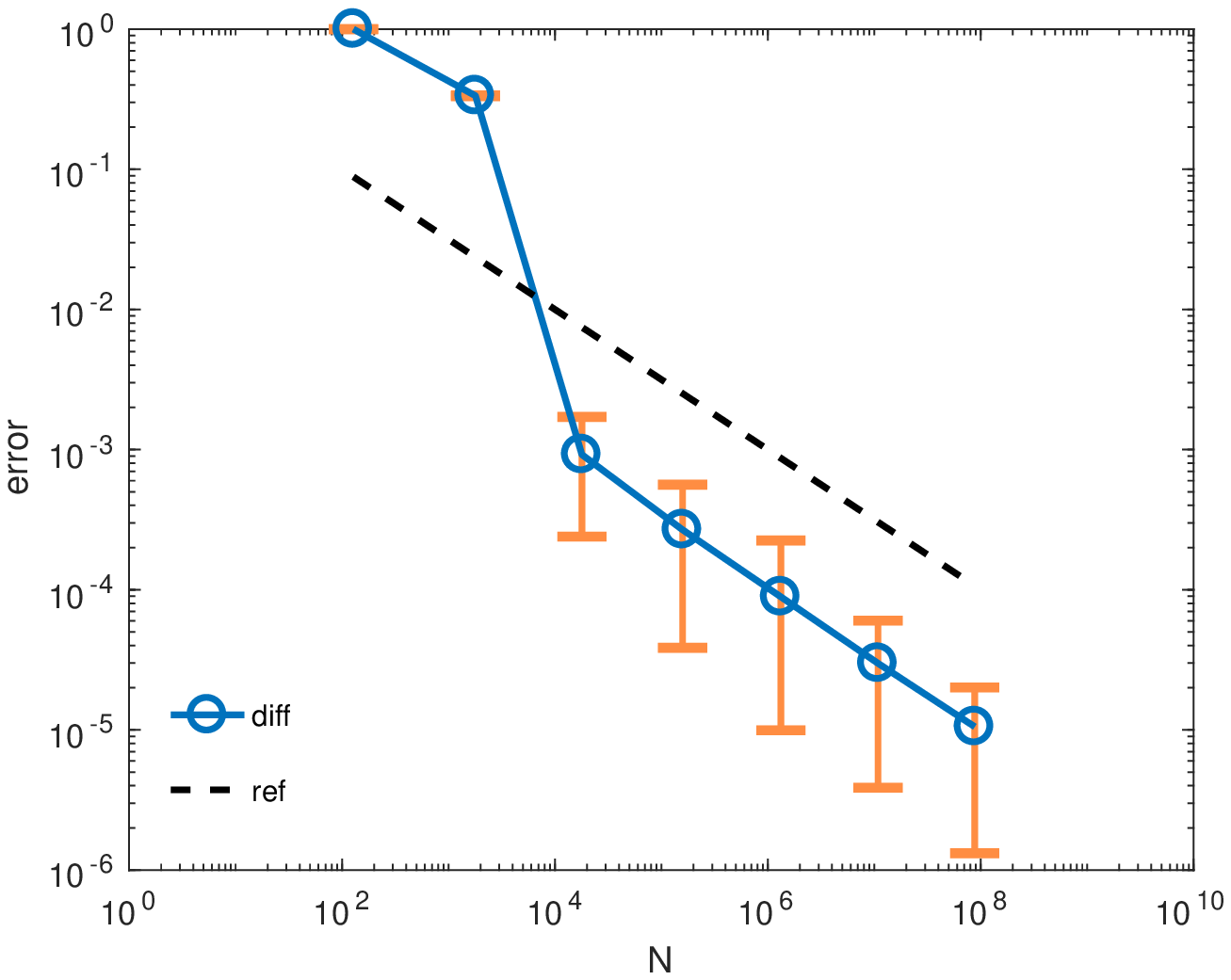}
\caption{Averaged relative errors or the multi-index algorithms with respect to the reference solution $G$  
         compared with the theoretical error bound $\mathcal{O}(C_N^{-1/2})$ (original algorithm (left) and symmetrized version (right)).
         Both plots shows the average error curve of $20$ runs of the algorithms as well as
         the empirical 90\%-confidence intervals of the computed error. 
         The symmetrized version reaches the accuracy of the non-symmetric version already for $N=6$ instead of $N=9$.}
 \label{fig:convMI2d}
\end{figure}

\subsection{Local mesh refinement}
The regularity of the exact solution can also be reduced by 
     re-entrant corners with corresponding reduced rates of FE convergence for quasi uniform meshes.
As is well-known (e.g. \cite{BacutaNistorZikatanov2005,AdlerNistor2015}), 
in two space dimensions, this is due to point-singularities in the solution.
These can be compensated by a-priori local mesh-refinement in $D$. 
Using hierarchies of so-called graded or suitable bisection-tree meshes,
and expressing regularity of solutions in terms of weighted $H^2(D)$ spaces,
the present regularity and FE convergence analysis remains valid verbatim,
with full convergence rates (see Section~\ref{sec:RedReg} for details).

This is demonstrated on the following example on the L-shaped domain $D:=[-1,1]^2\setminus (1,0)\times(-1,0)$ depicted in Figure~\ref{fig:meshLshape} with the same coefficient and PDE as in the previous example. However, as a right-hand side, we use $f=1$ and the quantity of interest is now defined by $G(u):=\int_{(0,1/2)^2} u\,dx$.
The graded meshes $\TT^\ell$ from Proposition~\ref{prop:RegKondr} are generated by newest vertex bisection by iteratively refining all elements $T$ which are coarser than the theoretically optimal grading of  
$\mathcal{O}({\rm dist}(\{0\},T)^{1/3}h_\ell)$. This results in a sequence of meshes with $\# (\TT^\ell) =\mathcal{O}(2^{2/3\ell})$. Figure~\ref{fig:meshLshape} shows one instance of this sequence of meshes. Figure~\ref{fig:meshsize} confirms the correct distribution of element diameters within the mesh.
\begin{figure}

\centering
\includegraphics[width=0.45\textwidth]{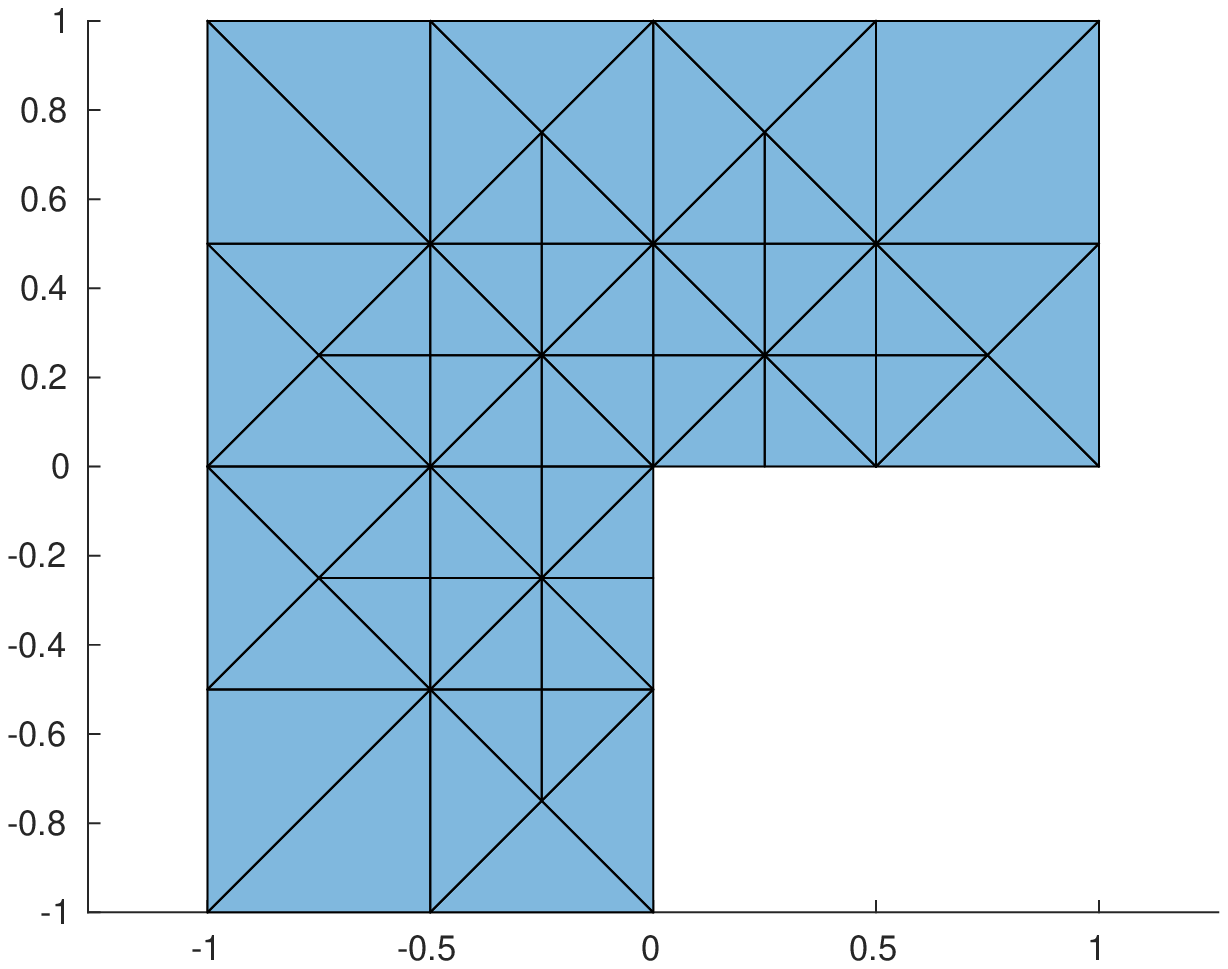}\qquad%
\includegraphics[width=0.45\textwidth]{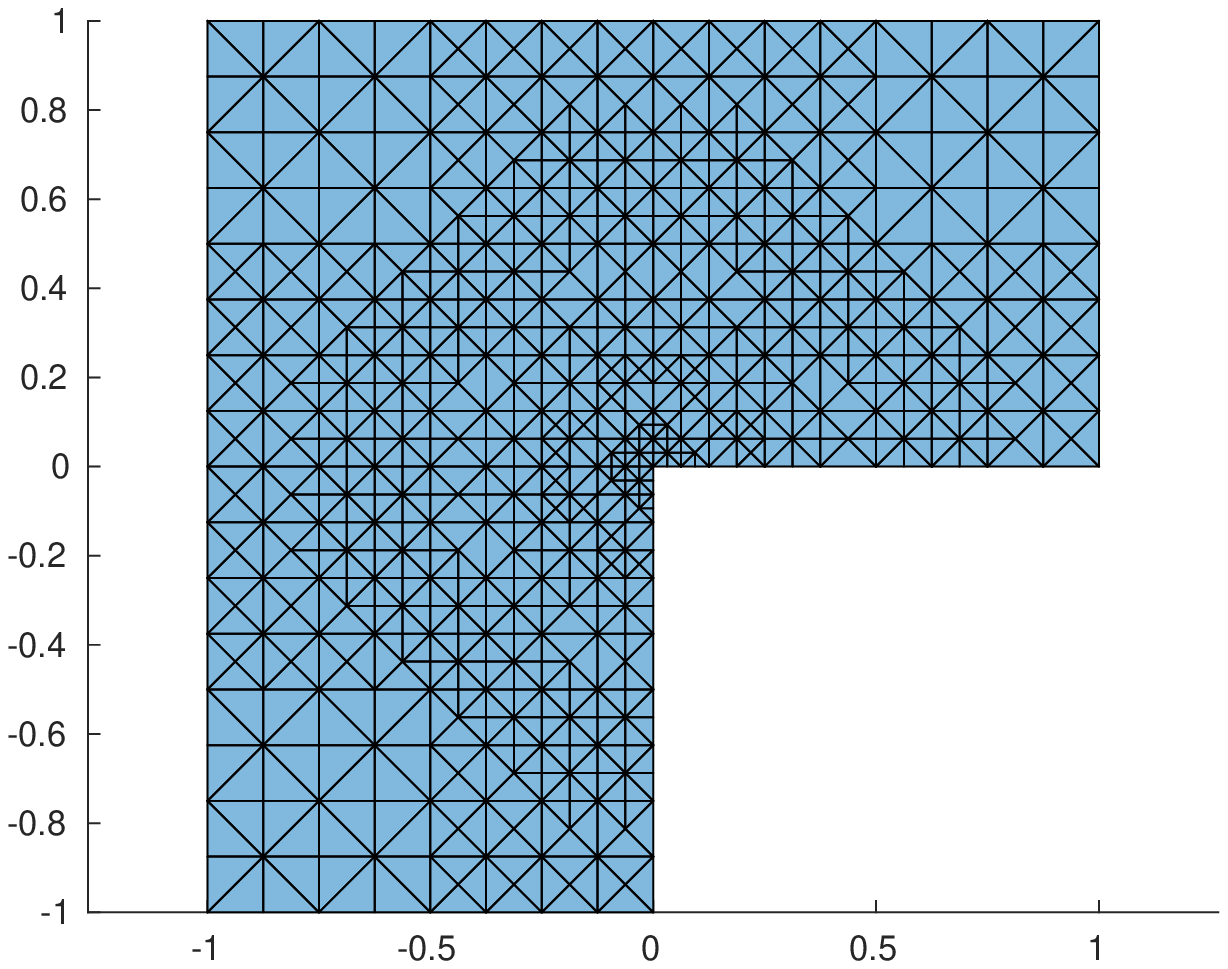}
 \caption{Two levels of graded mesh-refinement for the L-shaped domain.}
 \label{fig:meshLshape}
\end{figure}

\begin{figure}
\psfrag{elements}[cc][cc]{\tiny number of elements}
\psfrag{N}[cc][cc]{\tiny discretization parameter $N$}
\centering
\includegraphics[width=0.45\textwidth]{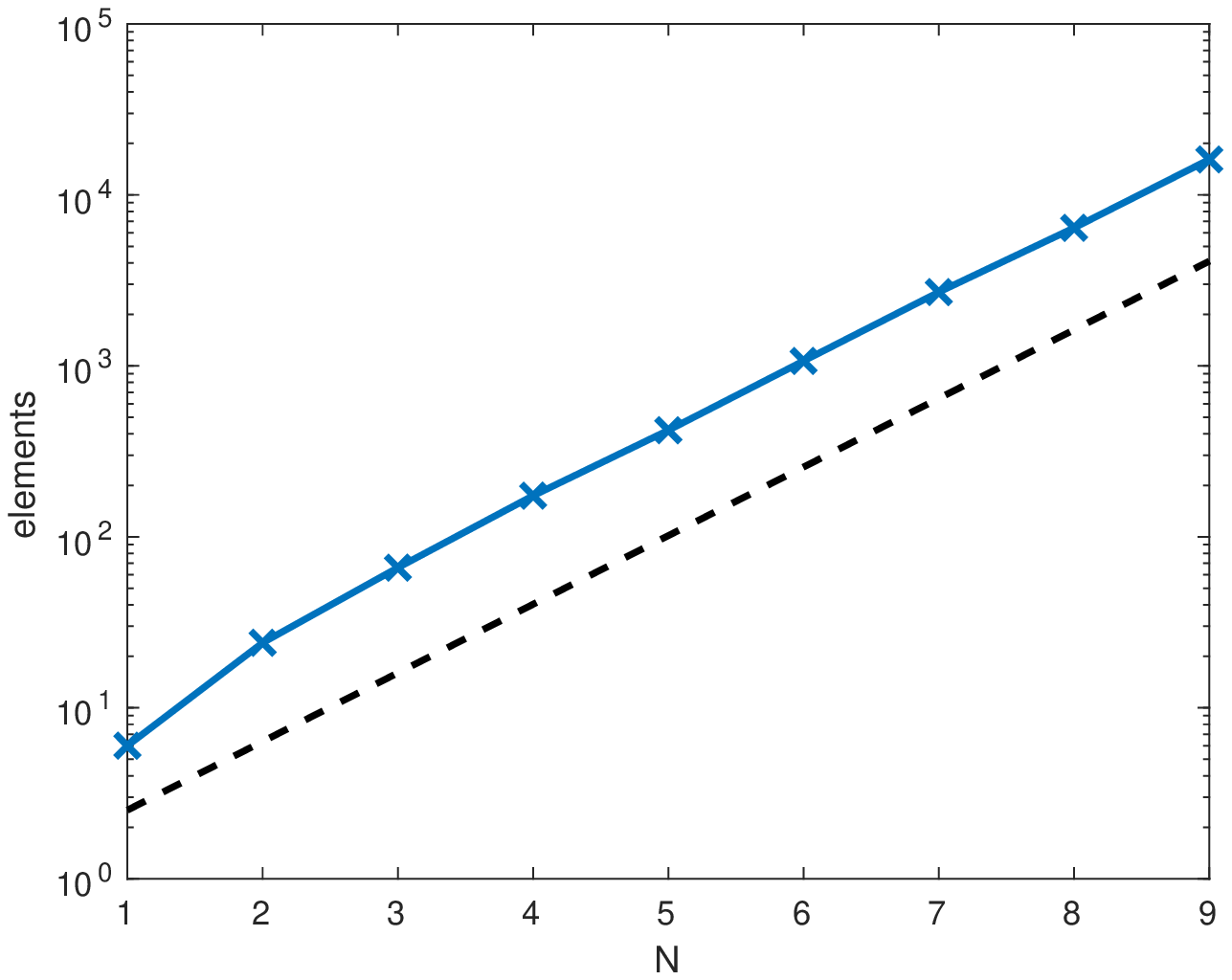}\hspace{5mm}%
\psfrag{dist}[cc][cc]{\tiny distance from origin}%
\psfrag{elsize}[cc][cc]{\tiny element diameter}%
\includegraphics[width=0.5\textwidth]{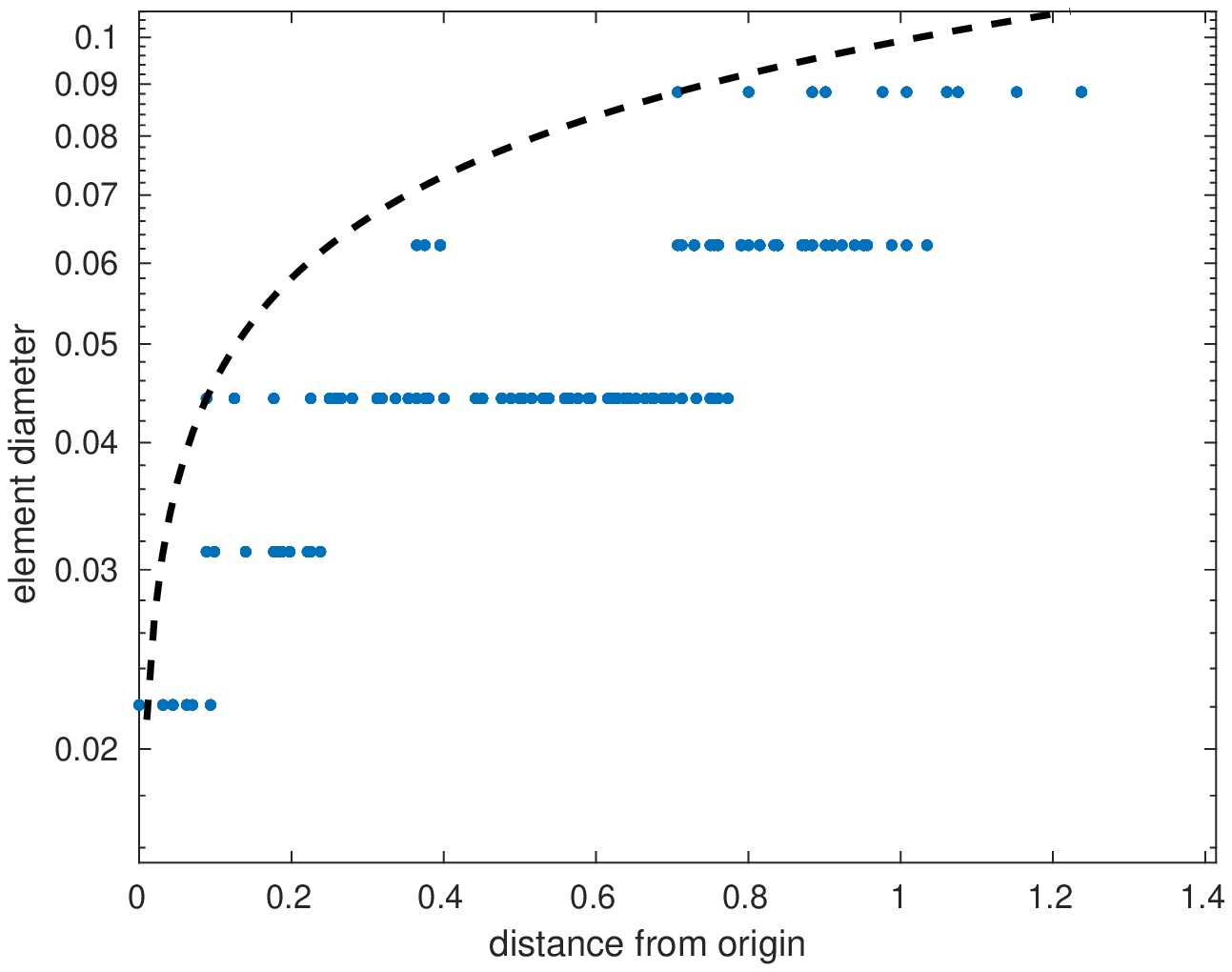}
 \caption{We see statistics of several graded meshes for levels $N=1,\ldots,9$. The left-hand side plot shows that the number of elements behaves as $\mathcal{O}(2^{2/3N})$. The right-hand side plot shows for the mesh $\TT^8$ that the distribution of element diameters with respect to their distance to the singularity behaves like $\mathcal{O}({\rm dist}(\{0\},T)^{1/3}h_N)$, where $h_N$ is the maximal element diameter.}
 \label{fig:meshsize}
\end{figure}

The performance of the multi-index Monte Carlo method is shown in Figure~\ref{fig:convMILshape2d} for the symmetrized version. Since we aim for the full convergence rate $\mathcal{O}(2^{-2N})$ in this example, we choose the level parameters $m_j:=8/3j$ as well as $s_\nu=\lceil 2^{2\nu/(3r-3)}\rceil$. Due to the much higher number of Monte-Carlo samples required in this example, we only performed four Monte-Carlo runs and show the averaged error in Figure~\ref{fig:convMILshape2d}. 
We observe optimal convergence behavior despite the presence of corner singularities 
in the exact solution. 
As a reference solution, we use the approximation on the next higher level~$N=14$.
\begin{figure}
\centering
\psfrag{error}[cc][cc]{\tiny relative error}
\psfrag{N}[cc][cc]{\tiny cost of multi-index FEM $C_N$}
\psfrag{diff}{\tiny $|G-G_N|/G$}
\psfrag{ref}{\tiny $\mathcal{O}(C_N^{-1/2})$}
\psfrag{diff}{\tiny $|G-G_N^{\rm symm}|/G$}%
\hspace{5mm}
\includegraphics[width=0.65\textwidth]{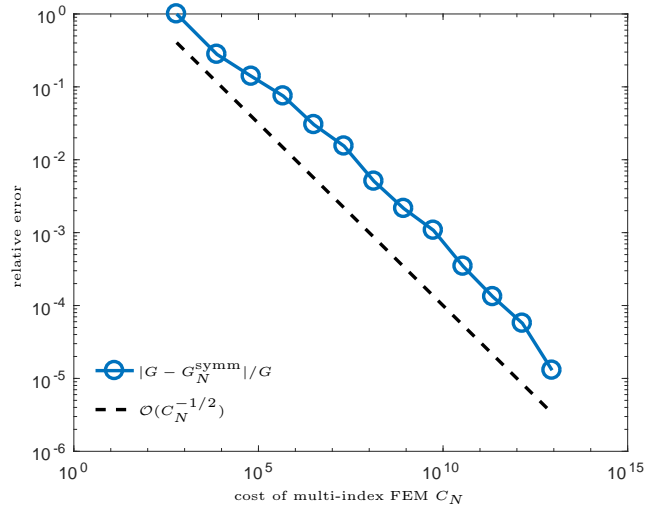}
\caption{Averaged relative errors or the multi-index algorithms with respect to the reference solution $G$  
         compared with the theoretical error bound $\mathcal{O}(C_N^{-1/2})$.
         The error curve is the average of four Monte-Carlo runs.}
 \label{fig:convMILshape2d}
\end{figure}
\section{Conclusion}
\label{S:Concl}
The present work shows that the multi-index Monte Carlo algorithm
with the indices being the discretization parameters of the finite element method, 
of the Monte Carlo method, and 
of the approximation of the random field is superior to its multi-level counterpart. 
The error estimates are rigorous and the product error estimate from Theorem~\ref{thm:proderr} 
might be of independent interest. 
The method can be combined with existing multi-index techniques which focus on 
sparse grids in the physical domain $D$ to further reduce the computational effort
under the provision of appropriate extra regularity.

\bibliographystyle{siamplain}
\bibliography{literature}
\end{document}